\documentclass[a4paper,12pt]{amsart}
\usepackage [latin1]{inputenc}
 \usepackage{amssymb,amsthm}
\usepackage{amscd} % allows index generation
\usepackage{graphicx,color} % standard LaTeX graphics tool
\usepackage{url} % <-- Para p\'aginas web o similar: \url{...}

\newtheorem{theorem}{Theorem}[section]
\newtheorem{corollary}[theorem]{Corollary}
\newtheorem{proposition}[theorem]{Proposition}

\theoremstyle{definition}
\newtheorem{definition}[theorem]{Definition}

\newtheorem{remark}[theorem]{Remark}

\def\r{\mathbb R}

\def\s{\mathbb S}

\begin{document}

\title[Axisymmetric stationary surfaces for the moment of inertia]{Axisymmetric stationary surfaces for the moment of inertia}
\author{Ulrich Dierkes}
\address{Fakult\"{a}t für Mathematik, Universit\"{a}t Duisburg-Essen, Thea-Leymann-Str. 9, 45147 Essen, Germany}
\email{ulrich.dierkes@uni-due.de}
\author{Rafael L\'opez}
\address{ Departamento de Geometr\'{\i}a y Topolog\'{\i}a\\ Universidad de Granada. 18071 Granada, Spain}
\email{rcamino@ugr.es}


\begin{abstract}  We investigate axisymmetric  surfaces  in Euclidean space that are stationary for the energy $E_\alpha=\int_\Sigma |p|^\alpha\, d\Sigma$. By using a phase plane analysis, we classify these surfaces when  they intersect orthogonally the rotation axis.  We also give some applications of the maximum principle characterizing the closed stationary surfaces and the compact stationary surfaces with boundary a circle when $\alpha=-2$. Finally, we prove that     helicoidal stationary surfaces must be rotational surfaces.  

\end{abstract}


\maketitle
%%%%%%%%%%%%%%%%%%%%
\section{Introduction and statement of the results}
%%%%%%%%%%%%%%%%%%%%%%%%%%%

Euler investigated planar curves of constant density   with  the  least  powers of the moment of inertia with respect to the origin $0$ \cite{eu}. In polar coordinates $r=r(s)$, these curves $\gamma$ are minimum of the energy $\int_\gamma r^\alpha\, ds$ where $\alpha\in\r$.   Explicit parametrizations of the extremals can be found  in \cite{ca,to}. When  $\alpha=2$, the energy represents the moment of inertia with respect to $0$ and the extremals are given by $ar^3=\mbox{sec}(3\theta+b)$, $a,b\in\r$.   For this value of $\alpha=2$,  Mason found the minimizers joining two given points of the plane \cite{ma}. 

 Recently, the first author and G. Huisken have generalized the Euler's problem in arbitrary dimensions  \cite{dh}. For surfaces in Euclidean $3$-space $\r^3$, let $\Sigma$ be a  connected, oriented surface   and consider a smooth immersion   $\Phi\colon \Sigma\to\r^3$ in   $\r^3$.  If $\alpha\in\r$, define the energy   
\begin{equation*}
E_\alpha[\Sigma]=\int_\Sigma |p|^\alpha\, d\Sigma,
\end{equation*}
where $d\Sigma$ denotes the area element induced on $\Sigma$ by the Euclidean metric $\langle,\rangle$ and  we are identifying $p\in\Sigma$ with its image $\Phi(p)$ by the immersion.   In order to use the techniques of calculus of variations, we need to assume that   the origin $0$ of $\r^3$ does  not belong to $\Sigma$. The first variation of $E_\alpha$ under compactly supported variations of $\Sigma$ is 
\begin{equation}\label{e0}
E_\alpha'(0)=\int_\Sigma \left( \alpha\frac{\langle\nu(p),p\rangle}{|p|^2}-H(p)\right)|p|^\alpha\xi(p)\, d\Sigma,
\end{equation}
where $H$ and $\nu$ are the mean curvature and the unit normal vector of $\Sigma$, respectively, and $\xi$ is the normal component of the variational vector field of the variation. The convention for the mean curvature is that $H$ is the sum of the principal curvatures. Thus a sphere of radius $r>0$ has $H=2/r$ with respect to the inward orientation. From \eqref{e0}, we characterize the critical points of $E_\alpha$ giving the following definition.

\begin{definition} A surface $\Sigma$ of $\r^3-\{0\}$ is said to be a {\it stationary surface} of $E_\alpha$ if 
\begin{equation}\label{eq1}
H(p)=\alpha\frac{\langle \nu(p),p\rangle}{|p|^2},\quad p\in\Sigma.
\end{equation}
\end{definition}
 If $\alpha=0$, the energy $E_0$ is     the area functional, where stationary surfaces are minimal surfaces.  We will discard the case $\alpha=0$.

In \cite{dh} the authors have studied the  stability of spheres and minimal cones as well as  minimizers of $E_\alpha$. See also \cite{d4} and generalizations to other energies in \cite{cx}.   Such as it was pointed in \cite{dh}, stationary surfaces of $E_\alpha$  can be viewed as  critical points of a weighted area defined in $\r^3$ where the density is $|p|^\alpha$. This implies that  stationary surfaces  \eqref{eq1}   are weighted minimal surfaces for this density. In particular, Eq.  \eqref{eq1} satisfies the maximum principle. 

The purpose  of this paper is the study and classification of  all axisymmetric  stationary surfaces of the energy $E_\alpha$. By the way, we revisit the results of \cite{dh} for closed stationary surfaces, we investigate  problems on compact stationary surfaces with boundary and finally we classify all helicoidal stationary surfaces.  

The paper is organized as follows.  In Sect.  \ref{s2} we give explicit examples of $\alpha$-stationary surfaces, namely, vector planes (for all $\alpha$) and spheres (only for $\alpha=-2,-4$). Using the tangency principle, we prove  that the only closed $\alpha$-stationary surfaces are spheres centered at the origin (Thm. \ref{t32}). This completes the initial classification given in \cite{dh}. In Sect. \ref{s23} we give more applications of the tangency principle for compact stationary surfaces with boundary. For $\alpha=-2$, we  characterize in Thm. \ref{t28} that spherical caps are the only compact $-2$-stationary surfaces whose boundary is a circle.

Sections \ref{s3}, \ref{s4} and \ref{s5}  constitute the main core of this paper. The purpose is the study of  axisymmetric stationary  surfaces obtaining a classification in case that the surface meets the rotation axis. First, we study   the relationship between the axis of rotation and the origin of coordinates.   Notice that in Eq. \eqref{eq1} the origin is a privileged point because the norm $|p|$ in the definition of the energy $E_\alpha$  is computed from $0$. We prove that except spheres containing the origin ($\alpha=-4$), the rotation axis must accross $0$ (Prop. \ref{pr2}).  In Sect. \ref{s4},  we prove existence of rotational stationary  surfaces intersecting the rotation axis. Writing Eq. \eqref{eq1} in radial coordinates, Eq.  \eqref{eq1} presents a singularity when the surface meets orthogonally the rotation axis.  Using the Banach fixed point theorem, we prove existence of such solutions (Thm. \ref{t1}). We also give a result of removability of single points for Eq. \eqref{eq1} (Thm. \ref{t-remove}). 

Section \ref{s5} is devoted to get  a classification of axisymmetric stationary  surfaces giving a geometric description when these surfaces meet the rotation axis: see  Thms. \ref{t2}, \ref{t3} and \ref{t5}. An overview of this classification is the following. 
\begin{enumerate}
\item If $\alpha>0$, then the surface is an entire graph.
\item If $\alpha\in (-2,0)$, then the surface is a graph on a plane outside a compact set around $0$. 
\item If $\alpha=-2$, then the surface is a sphere centered at $0$. 
\item If $\alpha<-2$, then the surface together the origin is a  closed surface non-smooth at $0$ except when $\alpha=-4$, where the surface is a sphere.  If $\alpha<4$, this surface is embedded.
\end{enumerate}

Going one step further, in Sect. \ref{s6} we ask which  are the stationary  surfaces of helicoidal type. We will prove that these surfaces are trivial in the sense that the pitch that determines the helicoidal motions must be zero, hence the surfaces simply   are surfaces of revolution (Thm. \ref{t61}). 
  
%%%%%%%%%%%%%%%%%%%%%%
\section{Examples of stationary surfaces and the tangency principle}\label{s2}
%%%%%%%%%%%%%%%%%%%%%

In the first part of this section, we see  the type of transformations of Euclidean space $\r^3$ that preserve stationary surfaces and then we show examples of stationary surfaces. 

 In general, rigid motions  of $\r^3$ do not preserve solutions of Eq. \eqref{eq1} because of the term $|p|^2$. This occurs, for example, with translations of $\r^3$ or symmetries about a plane which does not contain $0$.  The following result says that stationary surfaces are preserved by vector isometries and  by dilations from the origin.  

\begin{proposition}\label{pr0} 
\begin{enumerate}
\item Let $A\colon\r^3\to\r^3$ be a  vector isometry. If $\Sigma$ satisfies \eqref{eq1}, then $A(\Sigma)$ satisfies \eqref{eq1} for the same constant $\alpha$.  
\item Let $h\colon\r^3\to\r^3$ be a dilation from $0\in\r^3$.  If $\Sigma$ satisfies \eqref{eq1}, then $h(\Sigma)$ satisfies \eqref{eq1} for the same constant $\alpha$.
\end{enumerate}
\end{proposition}

We give explicit examples of stationary surfaces, focusing on isoparametric surfaces, that is, planes, spheres and cylinders. These surfaces are the only ones with constant Gauss curvature and constant mean curvature. Among these surfaces, we find those ones that are stationary surfaces of the energy $E_\alpha$. Since a plane of $\r^3$ has zero mean curvature, a plane $\Sigma$ is stationary for \eqref{eq1} if and only if $\langle\nu(p),p\rangle=0$ for all  $p\in\Sigma$. This implies that $0$ is a point of the plane hence, the plane is a vector plane. 

\begin{proposition}\label{pr00} 
A plane of $\r^3$ is an $\alpha$-stationary surface if and only if it is a vector plane. This occurs for all $\alpha\in\r$.
\end{proposition}

We now study which  spheres are stationary surfaces. 

\begin{proposition} \label{pr1}
The only stationary spheres  are the following:
\begin{enumerate}
\item  any sphere centered at $0$ ($\alpha=-2$); 
\item  any sphere containing $0$ ($\alpha=-4$).  
\end{enumerate}
\end{proposition}

\begin{proof}
Let $\Sigma$ be a sphere of radius $r>0$ centered at $q=(q_1,q_2,q_3)\in\r^3$. A parametrization of $\Sigma$ is 
$$\Phi(s,t)=(q_1+r\cos s\cos t,q_2+r\cos s\sin t,q_3+r\sin s),\quad s,t\in\r.$$
Since $H=2/r$, a substitution in \eqref{eq1} yields a polynomial of type
$$A_0(s)+\left((4+\alpha)rq_1\cos s\right)\cos t+\left((4+\alpha)rq_2\cos s\right)\sin t=0.$$
for some function $A_0$. Consequently,   
\begin{equation}\label{q1q2}
(4+\alpha)q_1 \cos s=(4+\alpha)q_2 \cos s=0
\end{equation}
 for all $s\in\r$. We have the following discussion of cases.
\begin{enumerate}
\item Case $\alpha=-4$. Now Eq. \eqref{eq1} is simply $|q|^2-r^2=0$. This is equivalent to say that the sphere contains the origin. 
\item Case $\alpha\not=-4$. Then \eqref{q1q2} implies $q_1=q_2=0$. Now Eq. \eqref{eq1} is 
$$2q_3^2+(2+\alpha)r^2+rq_3(4+\alpha)\sin s=0.$$
Since this holds for all $s\in\r$, we have $2q_3^2+(2+\alpha)r^2=0$ and $rq_3(4+\alpha)=0$. It follows that   $q_3=0$ and $\alpha=-2$.  Thus the center $q$ of the sphere  is the origin of $\r^3$.
\end{enumerate}
The converse is straightforward, that is, the spheres of items (1) and (2) are stationary surfaces. 
\end{proof}

Let us point out that   it is implicitly assumed that the origin is not included in the vector planes of Prop. \ref{pr00} neither in the spheres of Prop. \ref{pr1}. Finally we see that there are no stationary circular cylinders.

\begin{proposition}\label{pr3}
 No circular cylinders are stationary surfaces. 
\end{proposition}
\begin{proof} Let $\Sigma$ be a circular cylinder of radius $r>0$ and suppose that $\Sigma$ also is  an $\alpha$-stationary surface. Using (1) of Prop. \ref{pr1}, and after a vector isometry of $\r^3$, we can assume that the rotation axis of the cylinder is parallel to the $z$-axis and that the axis is included in the $xz$-plane. Then a parametrization of $\Sigma$ is 
$$\Phi(s,t)=(q_1+r\cos t ,q_2+r\sin t ,s),\quad s,t\in\r.$$
We know that $H=1/r$ and substituting in  \eqref{eq1} we have
$$q_1^2+q_2^2+r(\alpha +2) \left(q_1  \cos (t)+  q_2 r \sin (t)\right)+(1+\alpha)  r^2 +s^2=0.$$
This is a polynomial equation on $s$, whose leader coefficient is $1$, which it is not possible.   
\end{proof}
 
 We focus on closed $\alpha$-stationary surfaces. In   \cite[Thm. 1.6]{dh} the authors proved  the following result: 
\begin{enumerate}
\item If   $\alpha>-2$, then there are no closed stationary surfaces.
\item If  $\alpha=-2$,  then the only stable closed stationary surfaces are spheres centered at the origin. 
\item Let $\alpha<-2$. If $\Sigma$ is a (closed or not closed)  stationary surface, then its closure $\overline{\Sigma}$  must contain the origin $0$ of $\r^3$.  
\end{enumerate}
For the proof of (1) and (3), the authors used the maximum principle for the Laplacian of the function $p\mapsto |p|$. If $\alpha=-2$, the statement was proved using the expression of the second variation of the energy $E_\alpha$. We now revisit these results using a suitable statement of the maximum principle and proving that the condition on stability on (2) can be dropped.

We view the class of stationary surfaces of the energy $E_\alpha$ in the context of the  theory of manifolds with density. In $\r^3$, consider a positive density $\phi\in C^\infty(\r^3)$ used to weight volume and surface area. The new weighted volume and  weighted area are given by $dV_\phi=\phi dV_0$ and $dA_\phi=\phi dA_0$, where $dV_0$ and $dA_0$ are the Euclidean volume and area of $\r^3$. Let us observe that this is not equivalent to scaling the metric conformally by the factor $\phi$ because the exponent of $\phi$ in $dV_\phi$ and $dA_\phi$ is the same. A surface $\Sigma$ is a critical point of $A_\phi$  for any volume-preserving variation of $\Sigma$ if and only if its  {\it weighted mean curvature} $H_\phi$ is constant, where  $H_\phi$ is   defined by 
\begin{equation}\label{mw}
H_\phi= H- \langle\nu,\overline{\nabla}\phi \rangle,
\end{equation}
where $\overline{\nabla}$ is the gradient in $\r^3$ (\cite{ba,mo}). By choosing $\phi(p)=|p|^\alpha$, then $H_\phi=0$ coincides with  Eq. \eqref{eq1}. In conclusion, stationary surfaces of the energy $E_\alpha$ are weighted minimal surfaces ($H_\phi=0$) for the density $|p|^\alpha$.    Standard theory of elliptic equations can be   applied Eq. \eqref{mw}: see the general reference \cite{gt}. In the following result, we use the maximum principle to compare two surfaces that are tangent at some point    \cite{ps}. Denote by $\partial\Sigma$ the boundary of $\Sigma$ and let $\mbox{int}(\Sigma)=\Sigma-\partial\Sigma$ be the set of interior points of $\Sigma$.

\begin{proposition} Let $\Sigma_1$ and $\Sigma_2$ be two oriented surfaces tangent at some interior  point $p\in\Sigma_1\cap\Sigma_2$ such that   $\nu_1(p)=\nu_2(p)$. If $\Sigma_1$ lies above $\Sigma_2$ around $p$ according to $\nu_i(p)$, denoted by $\Sigma_1\geq \Sigma_2$, then $H_\phi^1(p)\geq H_\phi^2(p)$ (comparison principle). In addition,  if  $H_\phi^1=H_\phi^2=\mbox{constant}$,  then $\Sigma_1$ and $\Sigma_2$ coincide in an open set around  $p$ (tangency principle).  The same statement holds if $p$ is a common boundary point with the extra condition that $\partial\Sigma_1$ and $\partial \Sigma_2$ are tangent at $p$. 
\end{proposition}

If both surfaces satisfy $H_\phi=0$, that is, they are stationary surfaces for the same value of $\alpha$, and $\Sigma_1\geq \Sigma_2$ around a point, no matter what the orientation is,  because reversing the orientation, the equation  $H_\phi=0$ is preserved. We show some applications of the tangency principle.

In the theory of surfaces with constant mean curvature, the tangency principle is employed in the reflection technique. Alexandrov proved that spheres are the only embedded closed surfaces with constant mean curvature \cite{al}. However, this technique fails   for stationary surfaces because reflections about planes do not preserve the solutions of  \eqref{eq1} (see Prop. \ref{pr0}). Comparing with spheres centered at the origin, together the tangency principle, we can give an alternative proof of Thm. 1.6 of \cite{dh}, and improve the case $\alpha=-2$, where we drop the stability condition assumed there.  

\begin{theorem}\label{t32}
  Spheres centered at the origin are only closed $\alpha$-stationary surfaces.
\end{theorem}

\begin{proof}
Let $\Sigma$ be a closed $\alpha$-stationary surface.    Let $r>0$ be sufficiently big so $ \Sigma$ is included in the round ball determined by the sphere $\s^2(r)$ of radius $r>0$ centered at the origin. Letting $r\searrow 0$, we arrive until the first touching point   with $\Sigma$. Suppose that this occurs for $r=r_0$ and let $p\in\Sigma\cap \s^2(r_0)$. Consider the weighted mean curvature 
$H_\phi$ given in \eqref{mw}. With respect to the inward orientation on $\s^2(r)$, we have
$$H_\phi^{\s^2(r_0)}(p)=\frac{2}{r_0}-\alpha\frac{\langle\nu(p),p\rangle}{|p|^2}=\frac{2+\alpha}{r_0}.$$
For  $\Sigma$, we have $H_\phi^\Sigma=0$ regardless the orientation on $\Sigma$. Since $\Sigma\geq \s^2(r_0)$ around $p$, the comparison principle gives 
$$0\geq \frac{2+\alpha}{r_0}.$$
This yields  $\alpha\leq -2$. If $\alpha=-2$, the tangency principle implies that $\s^2(r_0)$ and $\Sigma$ coincide in an open set around $p$. By connectedness, we conclude $\s^2(r_0)=\Sigma$. 

Finally, we prove that there are no closed $\alpha$-stationary surfaces if $\alpha<-2$.  Let $r>0$ be sufficiently small  such that $\s^2(r)\cap\Sigma=\emptyset$. Let $r\nearrow\infty $ until the first touching point with $\Sigma$ at $r=r_1$.  With the inward orientation on $\s^2(r_1)$, we have $\s^2(r_1)\geq \Sigma$ around the contact point and the comparison principle yields
$$\frac{2+\alpha}{r_1}\geq 0.$$
This implies $\alpha\geq -2$, which it is a contradiction. 

\end{proof}

The last part in the above proof   can be used to conclude the following result: see also Thm. 1.6 in \cite{dh}.
\begin{corollary}\label{cor1}
Let  $\Sigma$ be a properly immersed $\alpha$-stationary surface. If $\alpha<-2$, then the closure of $\Sigma$ contains the origin $0\in\r^3$. 
\end{corollary}
%%%%%%%%%%%%%%%%%%%
\section{Applications of the tangency principle}\label{s23}
%%%%%%

Once the class of closed stationary surfaces is fully classified, it is natural to focus on compact stationary surfaces with boundary. As usually, if $\Gamma\subset\r^3$ is a curve, we say that a surface $\Sigma$ spans $\Gamma$ if there is an immersion $\Phi\colon\Sigma\to\r^3$ such that $\Phi_{|\partial\Sigma}\colon\partial\Sigma\to\Gamma$ is a diffeomorphism.  The simplest case of boundary is a circle. There are examples of compact $\alpha$-stationary surfaces spanning a circle: round discs in vector planes (for all $\alpha$) and spherical caps in spheres ($\alpha=-2,-4$).  As we will prove in Sect. \ref{s4}, there are rotational compact stationary surfaces spanning a circle for other values of $\alpha$. This is obtained by taking suitable pieces of rotational stationary surfaces that intersect orthogonally the rotation axis: see Fig. \ref{fig0}. We point out again, that  the reflection Alexandrov method cannot be applied to deduce that a compact stationary surface spanning a circle is a surface of revolution. 

 \begin{figure}[hbtp]
\begin{center}
\includegraphics[width=.4\textwidth]{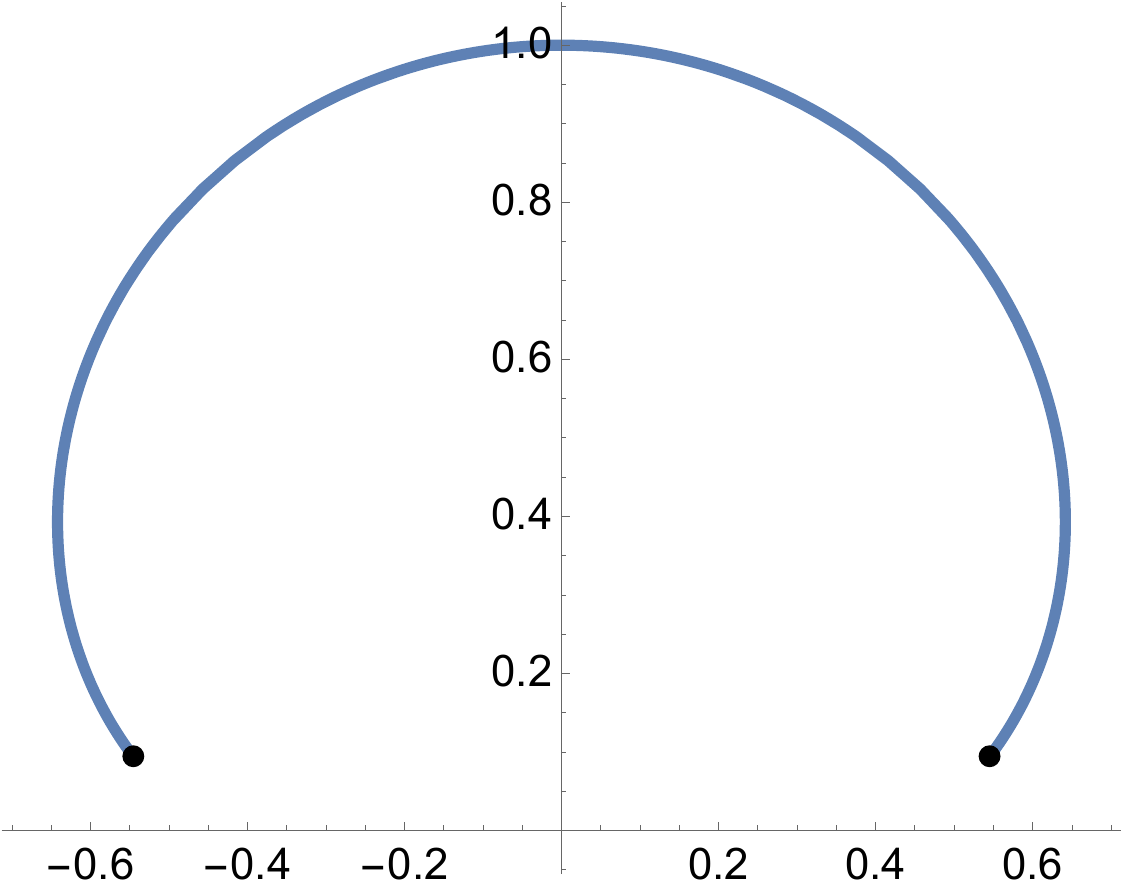}, \includegraphics[width=.35\textwidth]{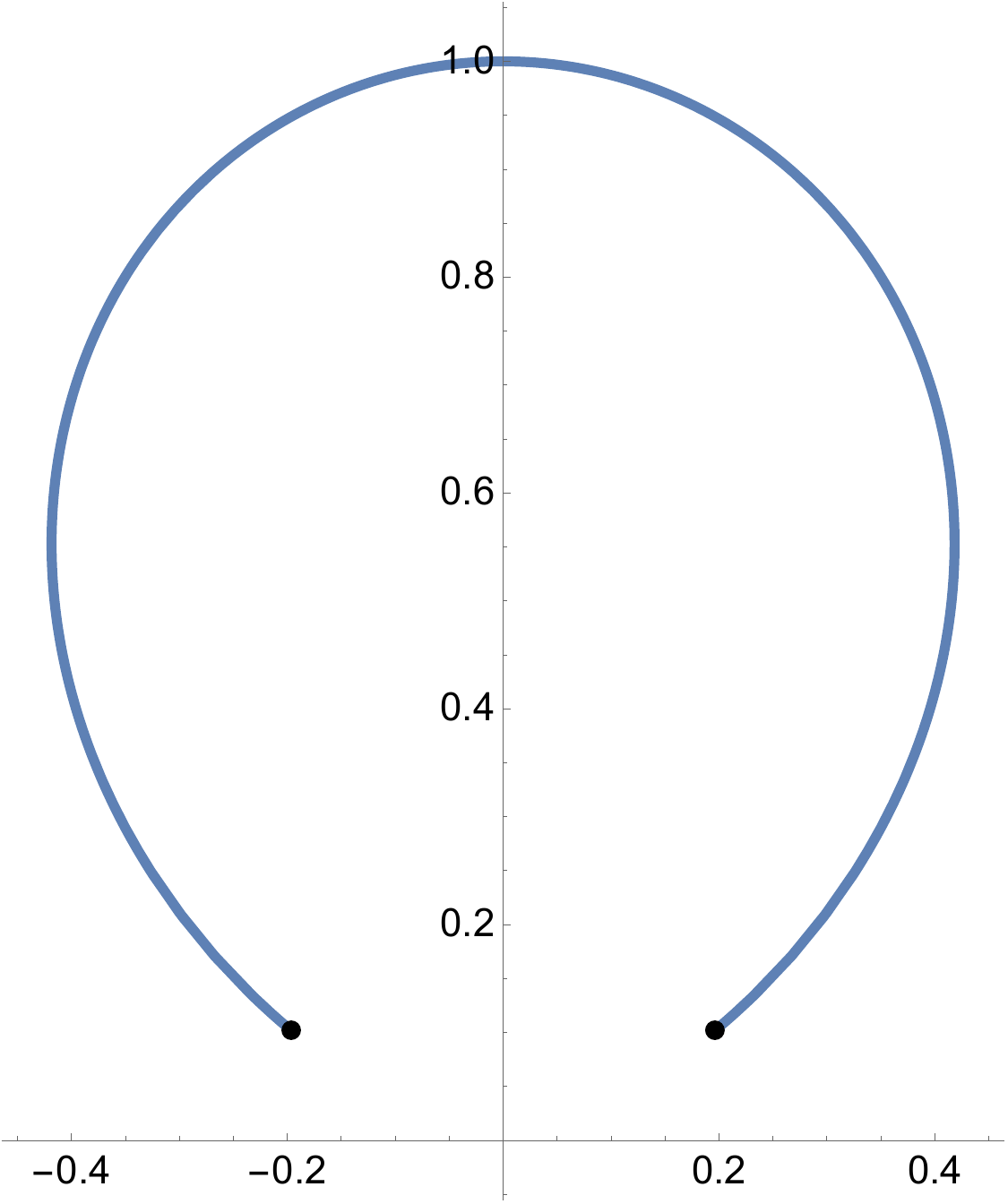}
\end{center}
\caption{Rotational stationary   surfaces whose boundary is a circle (black points). Here $\alpha=-3$ (left) and $\alpha=-5$ (right).}\label{fig0}
\end{figure}

When $\alpha=-2$, the circles contained in the spheres centered at $0$ are all invariant by rotations about an axis through $0$. However, if $\alpha=-4$, the circles contained in the spheres crossing $0$ are invariant by rotations but the rotation axis does not necessarily crosses $0$.

In a more general setting, we ask how the geometry of a given curve of $\r^3$ determines the shape of the possible stationary surfaces that spans.  In this sense, the comparison and tangency principle allows to control the shape of the surface in terms of the geometry of its boundary. An example is the following result.

\begin{proposition} Let $\Sigma$ be a compact $\alpha$-stationary surface whose boundary is a circle centered at the $z$-axis and contained in a horizontal plane. If $\alpha> -1$, then  $\Sigma$ contained in the vertical cylinder $D\times\r$, where $D$ is the  horizontal round disc  bounded by $\partial\Sigma$. 
\end{proposition}

\begin{proof} 
Let $C_r$ be a circular cylinder of radius $r>0$ whose rotation axis is the $z$-axis. With respect to the inward orientation on $C_r$, we have 
$$H_\phi^{C_r}(p)=\frac{1}{r}-\alpha\frac{\langle\nu(p),p\rangle}{|p|^2}=\frac{|p|^2+\alpha r^2}{r|p|^2}=
\frac{(1+\alpha)r^2+z^2}{r|p|^2}.$$
If $\alpha> -1$, then $H_\phi^{C_r}>0$.  Let $r$ be sufficiently big so $\Sigma$ is   included in the domain bounded by $C_r$. Letting $r\searrow 0$, consider the first cylinder $C_{r_0}$ that touches $\Sigma$ at some point $p\in\Sigma\cap C_{r_0}$. Suppose that $p$ is an interior point of $\Sigma$. Reversing the orientation on $\Sigma$ if necessary, and because the orientation on $C_{r_0}$ points inside, we have $\Sigma\geq C_{r_0}$ around $p$. The comparison principle yields $0\geq H_\phi^{C_{r_0}}$ but $H_\phi^{C_{r_0}}$ is positive which it is a contradiction. Therefore,  the first contact occurs at a boundary point of $\Sigma$.   Since $\partial\Sigma$ is a circle, and because $\partial\Sigma\subset C_{r_0}$, the  result is proved. 
\end{proof}

In Fig. \ref{fig0} we show examples of compact $\alpha$-stationary surfaces spanning a circle but the surfaces are not contained in the cylinder $D\times\r$. The value of $\alpha$ is less than $-1$.

Instead using cylinders, we can employ spheres $\s^2(r)$ centered at the origin.  This allows to prove the following result of uniqueness.

\begin{theorem}\label{t28}
If $\alpha=-2$,  spherical caps are the only compact stationary surfaces bounded by a circle whose rotation axis  crosses $0$.
\end{theorem}

\begin{proof} Let $\Gamma\subset\r^3$ be a circle whose rotation axis cross $0$. Then $\Gamma$ lies contained in a sphere $\s^2(r_1)$ of radius $r_1>0$. Let $\Sigma$ be a compact $\alpha$-stationary surface spanning $\Gamma$.  Following a similar argument as in the proof of Thm. \ref{t32}, let $r>0$ be sufficiently big so $\Sigma$ is included in the round ball determined by the sphere $\s^2(r)$. Letting $r\searrow 0$, we arrive until the first touching point between $\s^2(r_0)$ and $\Sigma$ for some $r_0>0$. In particular, $r_0\geq r_1$.  If the touching point is an interior point of $\Sigma$, then the tangency principle implies that $\Sigma\subset\s^2(r_0)$ and $r_1=r_0$. Since $\Gamma$ is a circle, then $\Sigma$ is a spherical cap of $\s^2(r_0)$.  If the touching point is a boundary point of $\Sigma$, and if the tangent planes of $\Sigma$ and $\s^2(r_0)$ coincide, then $\Sigma$ is a spherical cap  by the tangency principle again. In both cases, the result is proved.

Suppose now that we arrive until $r_1$ and the touching   point between $\Sigma$ and $\s^2(r_1)$ is not a tangent point.  In particular, $ \mbox{int}(\Sigma)$  lies contained in the round ball determined by $\s^2(r_1)$. We prove that this situation is not possible. Let now $r$ be sufficiently close to $0$ such that $\Sigma$ lies outside the ball determined by $\s^2(r)$. This is possible because $\Sigma$ is compact and $0\not\in\Sigma$. Letting $r\nearrow\infty$, we arrive to until the first touching point with $\Sigma$ at some $r=r_2$. Necessarily $r_2< r_1$ because we are assuming that $\mbox{int}(\Sigma)\cap\s^2(r_1)=\emptyset$.  Since the contact between $\Sigma$ and $\s^2(r_1)$ occurs at some interior point and both surfaces are stationary for $\alpha=-2$, the tangency principle implies that $\Sigma\subset\s^2(r_2)$ and $\Sigma$ is a spherical cap of $\s^2(r_2)$. This is a contradiction because   $r_1\not= r_2$. 
\end{proof}

In the theory of elliptic equations, the maximum principle is  also used to ensure uniqueness of the Dirichlet problem. However this is not true in general for Eq. \eqref{eq1} because translations of $\r^3$ do not preserve solutions of \eqref{eq1}.  Rather   graphs on planar domains, radial graphs are suitable in this type of arguments because solutions of \eqref{eq1} are preserved by dilations. 

\begin{proposition}  Let $\Omega$ be a domain of the unit sphere $\s^2$ and let $\varphi\in C^\infty(\partial\Omega)$ with $\varphi>0$. Then there is  at most one stationary radial graph on $\Omega$ with boundary data $\varphi$. 
\end{proposition}
 
 \begin{proof}
 Let $\Gamma$ be the radial graph of $\varphi$. Suppose $\Sigma_1$ and $\Sigma_2$ are two stationary radial graphs on $\Omega$ with $\partial\Sigma_1=\partial\Sigma_2=\Gamma$. Let $h_t$ be the dilation of $\r^3$ of ration $t>0$ and let $\Sigma_2^t=h_t(\Sigma_2)$. Notice that $\Sigma_2^t$ is an $\alpha$-stationary surface by Prop. \ref{pr0}. For $t$ sufficiently big, we have $\Sigma_2^t\cap\Sigma_1=\emptyset$. Let $t\searrow 0$ until the first touching point with $\Sigma_1$ at $t=t_0$. We know that $t_0\geq 1$. We now use  similar arguments as in Thm. \ref{t28} to prove $t_0=1$. We omit the details.
 \end{proof}
 
We finish this section by revisiting Thm. \ref{t32} when $\alpha\geq -2$. We give a proof that it  only uses the divergence theorem, which it is weaker than the tangency principle  
 
\begin{proposition}
\begin{enumerate}
\item If $\alpha>-2$, there are not closed stationary surfaces. 
\item If $\alpha=-2$, spheres centered at $0$ are the only closed stationary surfaces.
\end{enumerate}
\end{proposition}

\begin{proof} For any surface $\Sigma\subset\r^3$,  the Laplacian of the restriction of the norm $|p|^2$   on $\Sigma$   satisfies
$$\Delta|p|^2=4+2H\langle\nu(p),p\rangle.$$
Suppose that $\Sigma$ is an $\alpha$-stationary surface. By using Eq. \eqref{eq1}, we obtain
\begin{equation}\label{delta}
\Delta|p|^2=4+2\alpha\frac{\langle\nu(p),p\rangle^2}{|p|^2}.
\end{equation}
\begin{enumerate}
\item Suppose that $\Sigma $ is a closed  surface and that  $\alpha>-2$. Using the divergence theorem in \eqref{delta} and because $\partial\Sigma=\emptyset$, we have 
$$0=\int_\Sigma \left(4+2\alpha\frac{\langle\nu(p),p\rangle^2}{|p|^2}\right)\, d\Sigma=2\int_\Sigma \frac{2|p|^2+\alpha \langle\nu(p),p\rangle^2}{|p|^2}\, d\Sigma.$$
If $\alpha\geq 0$, this is not possible. If $-2<\alpha<0$ and because  $\langle\nu(p),p\rangle^2\leq |p|^2$ and $\alpha<0$, we have $\alpha\langle \nu(p),p\rangle^2\geq\alpha|p|^2$ for all $p\in \Sigma$. This yields
$$0= 2\int_\Sigma \left(\frac{2|p|^2+\alpha \langle\nu(p),p\rangle^2}{|p|^2}\right)\, d\Sigma\geq 2(2+\alpha)\mbox{area}(\Sigma)>0,$$
obtaining a contradiction.
\item If $\alpha  -2$, the same computation   yields 
$$0= 2\int_\Sigma \frac{2|p|^2-2 \langle\nu(p),p\rangle^2}{|p|^2}\, d\Sigma\geq 0.$$
This implies  $\langle\nu(p),p\rangle^2=|p|^2$ for all $p\in\Sigma$, hence $\Sigma$ is a sphere centered at $0$. 
\end{enumerate}
\end{proof}
    
%%%%%%%%%%%%%%%%%%%%%%%%%
\section{The rotation axis of a  stationary rotational surface}\label{s3}
%%%%%%%%%%%%%%%%%%%%%%%%

In the next sections, we investigate axisymmetric stationary surfaces, that is, stationary surfaces invariant by a one-parametric group of rotations of $\r^3$.  In the definition of the energy $E_\alpha$, the origin $0\in\r^3$  is a privileged point because the moment of inertial is calculated with respect to $0\in\r^3$.   In principle there is not {\it a priori} a relation between the rotation axis of an axisymmetric stationary surface and the origin.   For example, if $\alpha=-4$, we know that any sphere   containing the origin is stationary and this sphere is a surface of revolution with respect to any line through its center which may not cross $0$.  However, in the next result we prove that this the only exception. 

\begin{proposition} \label{pr2}
Let $\Sigma$ be an axisymmetric surface about an axis $L$. If $\Sigma$ is stationary, then either $0\in L$   or   $\Sigma$ is a sphere containing $0$.
\end{proposition}

\begin{proof} Using Prop. \ref{pr0} and after a vector isometry of $\r^3$, we can assume that the rotation axis $L$  is parallel to the $z$-axis and contained in the $xz$-plane coordinate. Thus the equations of $L$ are $\{x=q_1, y=0\}$, with $q_1\in\r$. A parametrization of $\Sigma$ is 
$$\Phi(s,t)=(q_1+x(s)\cos t, x(s)\sin t,z(s)),\quad s\in I, t\in\r,$$
where  $\gamma(s)=(x(s),0,z(s))$, $s\in I\subset\r$,  is the generating curve of $\Sigma$.   We need to prove  that $q_1=0$ (and thus $0\in L$) or that  $\Sigma$ is a sphere with $0\in\Sigma$. 

 Without loss of generality,  we  suppose  that $\gamma$ is parametrized by arc-length.  Since $x'(s)^2+z'(s)^2=1$, there is a smooth function $\psi$ such that
\begin{equation*}
\begin{split}
x'(s)&=\cos\psi(s),\\
z'(s)&=\sin\psi(s).
\end{split}
\end{equation*}
 We calculate all terms of the stationary surface equation \eqref{eq1}. The unit normal vector of $\Sigma$ is 
 $$\nu= (-\sin\psi\cos t,-\sin\psi\sin t,\cos\psi).$$
 The principal curvatures of $\Sigma$ are 
 \begin{equation}\label{kk}
 \kappa_1=\psi',\quad \kappa_2=\frac{\sin\psi}{x},
 \end{equation}
  and thus the mean curvature is 
 $$H=\psi'+\frac{\sin\psi}{x}.$$
 Equation \eqref{eq1} becomes
 $$A_0(s)+A_1(s) \cos t=0,$$
 where
 \begin{equation}\label{u3}
 \begin{split}
  A_1 &=xq_1(2x\psi'+(2+\alpha)\sin\psi)\\
  A_0&=\sin \psi  \left(q_1^2+(\alpha +1) x^2+z^2\right)+x \psi ' (q_1^2+x^2+z^2)-\alpha  x z \cos \psi .
  \end{split}
  \end{equation}
 Therefore we have $A_0(s)=A_1(s)=0$ for all $s\in I$.  From the equation $A_1=0$, we have two possibilities. If $q_1=0$, then the rotation axis $L$ is the $z$-axis which contains the origin. This proves the result in the first case. 
 
 Suppose $q_1\not=0$. Then  $A_1=0$ implies $2x\psi'+(2+\alpha)\sin\psi =0$ identically in $I$. This gives
\begin{equation}\label{s1}
\psi'=-\frac{(2+\alpha)\sin\psi}{2x}.
\end{equation}
Moreover, by substituting the value of $\psi'$ in the expression of $A_0$, we have 
\begin{equation}\label{ss2}
\sin\psi (q_1^2-x^2+z^2)+2 x z \cos\psi =0.
\end{equation}
 We solve this equation. Since our study is local, let us  write $\gamma$ as a graph on the $z$-axis. First we need to ensure that $\sin\psi\not=0$. In case  that  $\sin\psi=0$ identically, then $z=z(s)$ is a constant function. This implies that $\gamma$ is a horizontal line and thus $\Sigma$ is a horizontal plane. Since $\Sigma$ is stationary, this plane is the plane of equation $z=0$. In particular, this surface is rotational about the $z$-axis, proving the result in this particular case.
 
Finally, suppose  $\sin\psi\not=0$. Then we write $\gamma$ as the graph of a function $u$ on the $x$-axis by letting  $r=x(s)$, $u=z(s)$ and $u=u(r)$. Then we have   $u'=\cot\psi$. We will prove that $\Sigma$ is a sphere with $0\in\Sigma$.  Equation \eqref{ss2} is 
  $$2ruu'-u^2+r^2+q_1^2=0.$$ 
 The solution of this ODE is 
\begin{equation}\label{u2}
u(r)=\sqrt{q_1^2-r^2+rc},\quad c\in\r.
\end{equation}
Therefore, the generating curve of $\Sigma$ is $\gamma(r)=(q_1+u(r),0,r)$. It follows that $\gamma$ is a circle  contained in the $(x,z)$-plane of center $(q_1,\frac{c}{2})$ of radius $\frac{\sqrt{c^2+4q_1^2}}{2}$. In consequence, $\Sigma$ is a sphere containing $0$. This proves the result. If we follow further, we have  
 $$u''=-\frac{\psi'}{\sin^3\psi}.$$
 Since $\sin^2\psi=\frac{1}{1+u'^2}$, equation \eqref{s1} writes as 
 $$\frac{u''}{1+u'^2}=\frac{2+\alpha}{2u}.$$
Substituting   the   value of $u(r)$ given in \eqref{u2}, we arrive to 
$(4+\alpha)(4q_1^2+c^2)=0$. This implies $\alpha=-4$ because $4q_1^2+c^2>0$.  
\end{proof}

\begin{remark}   In case that the rotation axis does not cross the origin, then $\alpha=-4$ and $\Sigma$ is a sphere. We will see in Sect. \ref{s5} that  there are non-spherical stationary   surfaces for $\alpha=-4$ which are axisymmetric about the $z$-axis.
\end{remark}

%%%%%%%%%%%%%%%%%%%%%%%
\section{Axisymmetric stationary surfaces intersecting the rotation axis}\label{s4}
%%%%%%%%%%%%%%%%%%%%%%%%%%%%%%%%%

In this section we prove the existence of axisymmetric stationary surfaces intersecting orthogonally the rotation axis. 
Let $\Sigma$ be an axisymmetric stationary  surface. By Prop. \ref{pr2}, we know that the rotation axis crosses $0$. After a linear isometry of $\r^3$ (Prop. \ref{pr0}) we can assume that the rotation axis is the $z$-axis. 
A parametrization of $\Sigma$ is 
$$\Phi(s,t)=( x(s)\cos t, x(s)\sin t,z(s)),\quad s\in I\subset\r, t\in\r,$$
where $\gamma(s)=(x(s),0,z(s))$ is the generating curve and $x'^2+z'^2=1$.   Computing again \eqref{eq1} or equivalently, considering the equation $A_0=0$ in \eqref{u3}, we have  
  \begin{equation}\label{eq20}
  \psi'+\frac{\sin\psi}{x} = \alpha\frac{  z\cos\psi-x\sin\psi}{x^2+z^2}.
  \end{equation}
 It is natural to expect the existence of solutions intersecting orthogonally the rotation axis. However,  equation   \eqref{eq20} is singular at $x=0$ because of the left hand-side, hence  standard arguments   do not assure existence of solutions. 
 
 The method that we use to prove the existence of these solutions is the fixed point theorem. Since we require that the intersection of $\gamma$ is orthogonal to the $z$-axis, we can assume that $\gamma$ is a graph on the $x$-axis. Thus we change the parametrization of the generating curve $\gamma$ by  $\gamma(r)=(r,0,u(r))$, where $u=u(r)$ is a smooth function  defined in an subinterval of $(0,\infty)$.  Now Eq.  \eqref{eq20} writes as
 \begin{equation}\label{eq22}
 \frac{u''}{(1+u'^2)^{3/2}}+\frac{u'}{r\sqrt{1+u'^2}}=  \frac{\alpha(u-ru')}{(r^2+u^2)\sqrt{1+u'^2}}.
 \end{equation}
Multiplying by $r$, this equation becomes
$$ \left(  \dfrac{r u'(r)}{\sqrt{1+u'(r)^2}}\right)'=\alpha\frac{r(u-ru')}{(r^2+u^2) \sqrt{1+u'^2}}.$$
This equation is singular at $r=0$.  Let $r_0\geq 0$ and consider the initial value problem
\begin{equation}\label{rot-r}
\left\{\begin{array}{ll}
\left(  \dfrac{r u'(r)}{\sqrt{1+u'(r)^2}}\right)'=\alpha\dfrac{r(u-ru')}{(r^2+u^2)\sqrt{1+u'^2}} ,&\mbox{ in $(r_0,r_0+\delta)$}\\
u(r_0)=u_0& \\
  u'(r_0)=0,&
\end{array}\right.
\end{equation}
where $u_0>0$. The following theorem proves the existence of axisymmetric stationary surfaces intersecting orthogonally the rotation axis. 

\begin{theorem}\label{t1}
For any $u_0>0$ the initial value problem (\ref{rot-r}) with $r_0=0$ has a solution $u\in C^2([0,R])$ for some $R>0$. Moreover, the solution depends continuously on the parameters $\alpha$ and $u_0$.
\end{theorem}

\begin{proof} Let $\r_0^+=\{r\in\r\colon r\geq 0\}$. Define the functions  $f:\r\rightarrow\r$ and $g:\r_0^+\times\r^2\rightarrow\r$  by
$$f(y)=\frac{y}{\sqrt{1+y^2}},\qquad g(x,y,z)=\frac{\alpha(y-xz)}{(x^2+y^2)\sqrt{1+z^2}}.$$
It is clear that a function $u\in C^2([0,R])$, for some $R>0$, is a solution of (\ref{rot-r}) if and only if
\begin{equation*}\left\{
\begin{split}
(r f(u'))'&=r g(r,u,u')\\
u(0)&=u_0\\
 u'(0)&=0.
\end{split}\right.
\end{equation*}
 By means of the inverse $f^{-1}$ and integrating the first equation, the function  $u$ is given by 
 $$u(r)=u_0+\int_0^rf^{-1}\left(\frac{1}{s}\int_0^s tg(t,u,u')\, dt\right)\, ds.$$
 Fix $R>0$ to be determined later. For $u\in C^1([0,R])$,  define the operator $\mathsf{T}$ by
$$
( \mathsf{T} u)(r)=u_0+\int_0^r f^{-1}\left( \int_0^s\frac{t}{s}g(t,u,u')\, dt\right)\, ds.
$$
It is clear that a fixed point $u$ of  $ \mathsf{T}$, $\mathsf{T}u=u$, is a solution of the initial value problem (\ref{rot-r}). The existence of the fixed point will be proved by using the Banach fixed point theorem. Let $C^1([0,R])$ endowed with  the norm $\|u\|=\|u\|_\infty+\|u'\|_\infty$. We will prove that $\mathsf{T}$ is a contraction mapping in a  closed ball about  $u_0$,  $\overline{\mathcal{B}(u_0,\epsilon)}\subset C^1([0,R])$ for some $\epsilon>0$.

We consider the next three steps.

\begin{enumerate}
\item Prove that     $ \mathsf{T}$ is well-defined. Since 
 $$f^{-1}\colon (-1,1)\to \r,\quad  f^{-1}(x)=\frac{x}{\sqrt{1-x^2}},$$ 
we need to check  
$$ \left|\int_0^s\frac{t}{s}g(t,u,u')\, dt\right|<1,\mbox{ for all $r\in [0,R]$}.$$
Let $\epsilon<\{1,u_0\}$. Consider the functions $f^{-1}$ and $g$ defined in $[-\epsilon,\epsilon]$ and  $[-\epsilon,\epsilon]\times [u_0-\epsilon,u_0+\epsilon]\times[-\epsilon,\epsilon]$, respectively. Let $M>0$ denote an upper bound of $g$ obtained as follows:
$$|g|\leq \frac{|\alpha|}{(u_0-\epsilon)^2}(u_0+\epsilon+\epsilon^2)\leq \frac{|\alpha|(u_0+2)}{(u_0-\epsilon)^2}:=M.$$
Let $R$ be a positive number such that $R<\frac{1}{M}$. For all $s\in [0,R]$, we have
$$\left| \int_0^s\frac{t}{s}g(t,u,u')\, dt\right|\leq \frac{Ms}{2}\leq \frac{MR}{2}<\frac{1}{2},$$
 which leads  that $T$ is well-defined.
 
\item Prove  the existence of $\epsilon>0$ such  that  $ \mathsf{T}(\overline{\mathcal{B}(u_0,\epsilon)})\subset \overline{\mathcal{B}(u_0,\epsilon)}$.  First, we change the value of $R$ by another one satisfying
\begin{equation}\label{r1}
R<\min\{\frac{1}{M},\frac{\sqrt{3}}{2}\epsilon,\frac{2 \epsilon}{M\sqrt{4+\epsilon^2}}\}.
\end{equation}
 Since $ f^{-1}$ is an increasing function

\begin{equation*}
|( \mathsf{T} u)(r)-u_0|\leq \int_0^r f^{-1}\left(\int_0^s\frac{t}{s} M\, dt\right)\, ds 
\leq \int_0^r f^{-1}\left(\frac{Ms}{2}\right)\, ds\leq  \frac{R}{\sqrt{3}}.
\end{equation*}
By using \eqref{r1}, we have
$|( \mathsf{T} u)(r)-u_0|\leq  \epsilon/2$. Thus $\| \mathsf{T} u-u_0\|_\infty\leq\frac{\epsilon}{2}$. Similarly, we have
\begin{equation*}
\begin{split}|( \mathsf{T} u)'-(u_0)'(r)|&\leq  f^{-1}\left(\left|\int_0^r\frac{t}{r}g(t,u,u')\, dt\right|\right)\leq  \left| f^{-1}\left(\frac{M}{2}r\right)\right|\\
&\leq \frac{MR}{\sqrt{4-M^2R^2}}\leq\frac{\epsilon}{2}
\end{split}
\end{equation*}
because \eqref{r1} again. Then $\|( \mathsf{T} u-u_0)'\|_\infty\leq  \epsilon/2$.
Definitively, we have proved $\|  ( \mathsf{T} u-u_0)\|\leq \epsilon$. 
\item Prove  that  $ \mathsf{T}\colon \overline{\mathcal{B}(u_0,\epsilon)}\to \overline{\mathcal{B}(u_0,\epsilon)}$ is a contraction.
 Denote $L_{ f^{-1}}$ and $L_g$ the Lipschitz constants of $ f^{-1}$ and $g$ in their respective domains $[-\epsilon,\epsilon]$ and $[-\epsilon,\epsilon]\times[u_0-\epsilon,u_0+\epsilon]\times[-\epsilon,\epsilon]$, respectively. For all $u,w\in C^1([0,R])$, we have
$$\| \mathsf{T} u- \mathsf{T}w\|=\| \mathsf{T} u- \mathsf{T}w\|_\infty+\|( \mathsf{T} u)'-( \mathsf{T}w)'\|_\infty,$$
Consider the first term $\| \mathsf{T} u- \mathsf{T}w\|_\infty$. Given $u,w\in \overline{\mathcal{B}(u_0,\epsilon)}$ and  $r\in [0,R]$, we have
\begin{equation}\label{c1}
\begin{split}
|( \mathsf{T} u)(r)-( \mathsf{T}w)(r)|&\leq L_{ f^{-1}} \left|\int_0^r \left(\int_0^r\frac{t}{s}(g(t,u,u')-g(t,w,w'))\, dt\right) \, ds\right|\\
 & \leq L_{ f^{-1}}L_g(\|u-w\|_\infty+\|u'-w'\|_\infty)\int_0^r(\int_0^s\frac{t}{s}\, dt)\, ds\\
 &= L_{ f^{-1}}L_g \frac{r^2}{4} \|u-w\| \leq L_{ f^{-1}}L_g \frac{R^2}{4} \|u-w\| .
\end{split}
\end{equation}
For the term   $\|( \mathsf{T} u)'-( \mathsf{T}w)'\|_\infty$, the argument is similar. Indeed,  
\begin{equation}\label{c2}
\begin{split}
|( \mathsf{T} u)'(r)-( \mathsf{T}w)'(r)| &\leq \left| f^{-1}\left(\int_0^r \frac{t}{r}(g(u,u')-g(w,w'))\, dt\right)\right|\\
&\leq  L_{ f^{-1}}L_g  \|u-w\|   \int_0^r  \frac{t}{r}\, dt= L_{ f^{-1}}L_g \frac{r}{2} \|u-w\|  \\
&\leq  L_{ f^{-1}}L_g   \frac{R}{2}\|u-w\|.
 \end{split}
\end{equation}
Let us change $R$ by the condition \eqref{r1} together  
$$R\leq \min\{\frac{1}{\sqrt{ L_{ f^{-1}}L_g}},\frac{1}{  2L_{ f^{-1}}L_g}\}.$$
 Then we find from \eqref{c1} and \eqref{c2}, $\|\mathsf{T} u-\mathsf{T}w\|_\infty\leq\frac14\|u-w\|$ and $\|(\mathsf{T} u)'-(\mathsf{T}w)'\|_\infty\leq\frac14\|u-w\|$, respectively. This gives
 $$\| Tu-Tw\|\leq \frac12\|u-w\|.$$
 This proves that $\mathsf{T}$ is a contraction.

\end{enumerate}
 
It remains to prove that the solution $u$ extends with $C^2$-regularity at $r=0$. Let $r\to 0$ in Eq. \eqref{eq22}. By L'H\^{o}pital rule, we have 
\begin{equation}\label{ii}
2u''(0)=\frac{\alpha}{u_0},
\end{equation}
hence $u''(0)=\alpha/(2u_0)$.  
\end{proof}

In Thm. \ref{t1}, we have proved existence of radial solutions of \eqref{rot-r} in   discs $B_R(0)\subset\r^2$ for sufficiently small radius $R$. By the way, we have proved regularity of solutions of \eqref{rot-r} in case that the solution  $u=u(r)$ arrives until $r=0$ under the condition $u'(0)=0$. Notice that the solutions of \eqref{rot-r} with initial conditions at $r_0>0$ may not extend until $r=0$. This is the case for $\alpha>-2$. In Fig. \ref{fig111}, we show the solutions of \eqref{rot-r} when $r_0=1$. The initial conditions are $u(1)=1$ and $u'(1)=0$ and the solutions cannot extend until $r=0$.  

\begin{figure}[hbtp]
\begin{center}
\includegraphics[width=.45\textwidth]{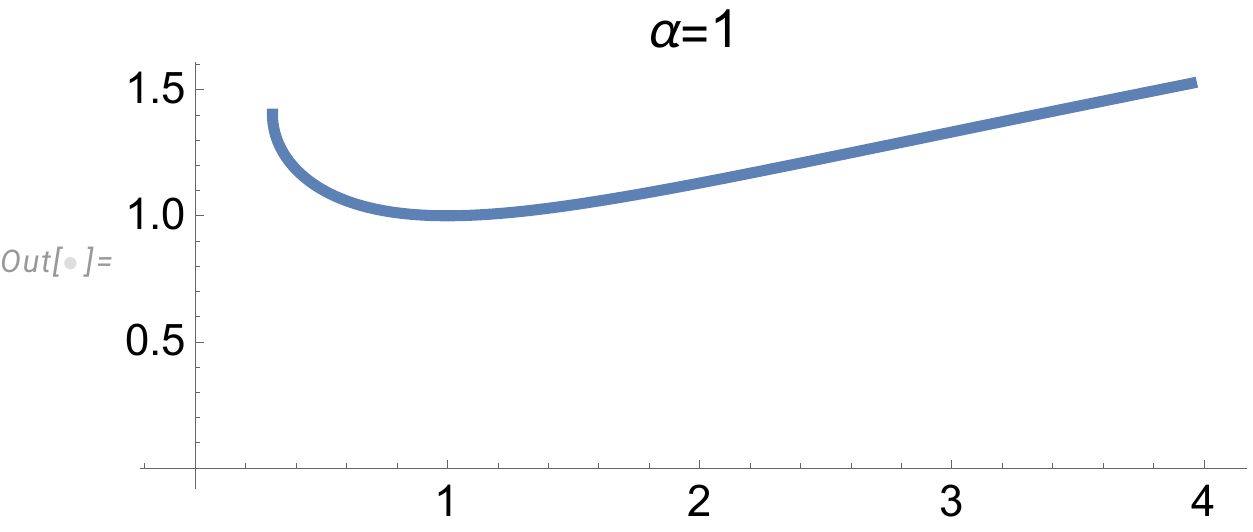} \quad \includegraphics[width=.5\textwidth]{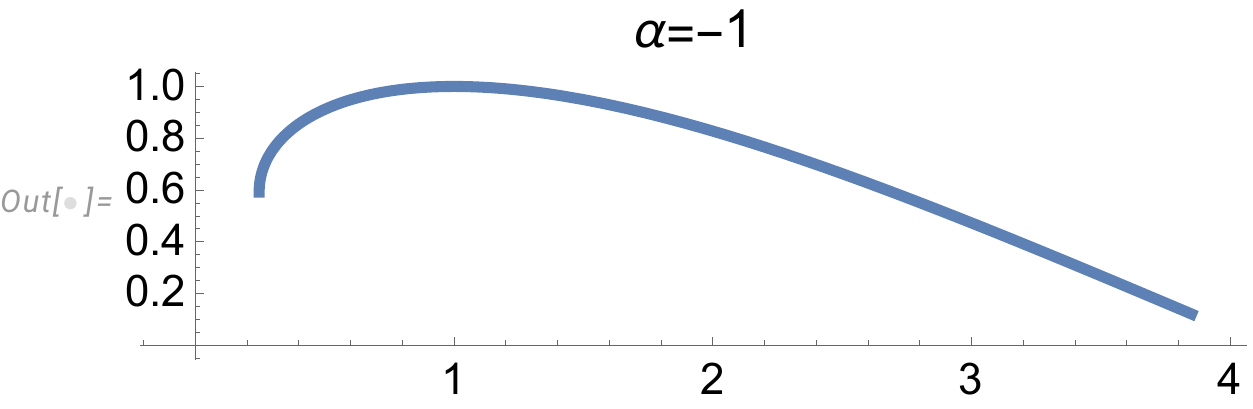} 
\end{center}
\caption{Two solutions of \eqref{rot-r} which cannot extend to $r=0$. The initial conditions are $r_0=u_0=1$ and $\alpha=1$ (left) and $\alpha=-1$ (right). }\label{fig111}
\end{figure}

In case that a solution of \eqref{rot-r} with $r_0>0$ can be extended to $r=0$,   regularity at $r=0$ is not assured.  In  the following theorem we address the problem of    removable singularities for Eq. \eqref{eq1}. Here we follow the general theory  for solutions of mean curvature equation type. The typical example is   the minimal surface equation when  the singularity consists in a single point. It is well known that under this situation, the  singular point can be removed \cite{be,se}.

We use the   notation $q=(x,y)\in\r^2$ for points of $\r^2$ and  $(\cdot)$ for the Euclidean product of $\r^2$. First, we write Eq. \eqref{eq1} in nonparametric way
\begin{equation}\label{eq-div}
\mbox{div}\left(\frac{Du}{\sqrt{1+|Du|^2}}\right)=\alpha\frac{u-(q\cdot Du)}{(x^2+u^2)\sqrt{1+|Du|^2}},
\end{equation}
where the graph of $u=u(x,y)$  is an $\alpha$-stationary surface. 
\begin{theorem}\label{t-remove}
 Let $u$ be a $C^2$-solution    of \eqref{eq-div} defined in a punctured disk $B_r(0)-\{0\}$  which is Lipschitz continuous on $B_r(0)$.    Then $u$ is analytic on $B_r(0)$. 
\end{theorem}

\begin{proof} 
We prove that $u$ is a weak Lipschitz-solution of \eqref{eq-div} on $B_r(0)$. Choose a sequence $\{\eta_n\}\in C_c^\infty(B_r(0))$ with the following properties:
\begin{enumerate} 
\item  $\mbox{supp}(\eta_n)\subset B_r(0)-\{0\}$, in particular, $\eta_n= 0$ identically in  a small neighborhood around $0$;
\item $0\leq \eta_n\leq 1$ and $\eta_n\to 1$ a.e. on $B_r(0)$;
\item $||D\eta_n||_{1,B_r}=\int_{B_r(0)}|D\eta_n|\, dxdy\to 0$ as $n\to\infty$.
\end{enumerate}
 Let $\varphi\in C_c^\infty(B_r(0))$ be arbitrary and put $\phi_n=\varphi\eta_n$, $n\in\mathbb{N}$. Then $\phi_n\in C_c^\infty(B_r(0)-\{0\})$
 and \eqref{eq-div} yields the relation 
 \begin{equation}\label{eq-div2}
 \int_{B_r(0)}\frac{Du\cdot D\phi_n}{\sqrt{1+|Du|^2}}\, dxdy=\alpha\int_{B_r(0)}\frac{((q\cdot Du)-u)\phi_n}{(x^2+u^2)\sqrt{1+|Du|^2}}\, dx dy.
 \end{equation}
 Observing that $D\phi_n=\eta_nD\varphi+\varphi D\eta_n$, we conclude from \eqref{eq-div2} and (2), (3) the relation
 $$\int_{B_r}\frac{Du\cdot D\varphi}{\sqrt{1+|Du|^2}}\, dx dy=\alpha\int_{B_r}\frac{(q\cdot Du)-u}{(x^2+u^2)\sqrt{1+|Du|^2}}\varphi\, dx dy.$$
 This means that   $u$   is a Lipschitz continuous solution of \eqref{eq-div} which, in addition, is of class $C^2$ on the punctured disc $B_r(0)-\{0\}$. By elliptic regularity theory,  this implies that $u$ is of class $C^2$ in $B_r(0)$ and hence analytic in $B_r(0)$. 

\end{proof}

As a consequence of this theorem, if a radial solution $u$ of \eqref{eq20} can be extended until $r=0$ and intersects the $z$-axis non-tangentially,  then this solution is analytic in the disk $B_{r_0}(0)$ and consequently, the intersection with the rotation axis must be orthogonal. In case that  $\gamma$ intersects the $z$-axis tangentially, one can show that the rotational graph $u$ of $\gamma$ is a weak solution of \eqref{eq-div} which is of class $  C^2(B_r(0)-\{0\})\cap H_1^1(B_r(0))$ \cite{d5}. 
 
%%%%%%%%%%%%%%%%%%%%%%%%%
\section{Geometry properties of axisymmetric stationary surfaces}\label{s5}
%%%%%%%%%%%%%%%%%%%%%%%%
 In this section we give a classification  of the axisymmetric solutions of Eq. \eqref{eq1} when the solution intersects orthogonally the rotation axis,  obtaining their main geometric properties.     Suppose that the generating curve $\gamma$ of the surface is parametrized by arc-length, $\gamma(s)=(x(s),0,z(s))$, $x'^2+z'^2=1$. By using \eqref{eq20},  the  ODE system that determines the generating curve $\gamma$ is 
  \begin{equation}\label{eq2}\left\{
 \begin{split}
 x'(s)&=\cos\psi(s),\\
z'(s)&=\sin\psi(s),\\
 \psi'(s) &= \alpha\frac{z(s)\cos\psi(s)-x(s)\sin\psi(s)}{x(s)^2+z(s)^2}-\frac{\sin\psi(s)}{x(s)}.
 \end{split}\right.
 \end{equation}
The initial conditions are  $x(0)=0$ and $\psi(0)=0$ (orthogonal condition). By Prop. \ref{pr0}, after a vector isometry we can suppose that $z(0)>0$, and after a dilation, let   $z(0)=1$.  Definitively, consider initial values at $s=0$  
 \begin{equation}\label{ini}
 \left\{
 \begin{split}
 (x(0),z(0))&=(0,1),\\
 \psi(0)&=0.
 \end{split}\right.
 \end{equation}
 
In many of the next arguments, we will analyze the phase plane of an autonomous system. For this, it is convenient to introduce the polar angle $\theta$ determined by the condition
 \begin{equation}\label{polar}
 \tan\theta=\frac{z}{x}.
 \end{equation}
By using the first two equations of \eqref{eq2}, we have 
 $$\theta'=\cos^2\theta\frac{\sin\psi-\tan\theta\cos\psi}{x}=\frac{\cos\theta}{x}\sin(\psi-\theta).$$
 On the other hand, the third equation in \eqref{eq2} can be written as
\begin{equation*}
\psi' =
 \alpha\frac{\tan\theta\cos\psi-\sin\psi}{x(1+\tan^2\theta)}-\frac{\sin\psi}{x}=-\frac{\alpha\cos\theta \sin(\psi-\theta)+\sin\psi}{x}.
\end{equation*}
Then 
\begin{equation*}
\frac{d\psi}{d\theta}=\frac{\psi'(s)}{ \theta'(s)}=-  \frac{\alpha \cos \theta \sin(\psi-\theta) +\sin\psi}{ \cos \theta \sin(\psi-\theta) }.
\end{equation*}
This gives the planar autonomous ordinary system
\begin{equation}\label{eq3}
\left\{
\begin{split}
\frac{d\psi}{dt}&:=h_1(\psi,\theta)= -\sin\psi-\alpha \cos \theta \sin(\psi-\theta),\\
\frac{d\theta}{dt}&:=h_2(\psi,\theta)=\cos \theta \sin(\psi-\theta).
\end{split}\right.
\end{equation}
We study this system according to the Bendixson-Poincar\'e theory.  From the equation $h_2(\psi,\theta)=0$, the equilibrium points of the  system \eqref{eq3} are of type $(n\pi,k\pi)$ and $(n\pi,\frac{\pi}{2}+k\pi)$, with $n,k\in\mathbb{Z}$. We need to distinguish if $n$ is even or not. Thus, the equilibrium points $(\psi,\theta)$ are of  the following three classes of points: 
\begin{equation*}
\begin{split}
 P_1&=(2n\pi,k \pi),\\
 P_2&=((2n-1)\pi,k\pi),\\
 P_3&=(n\pi,\frac{\pi}{2}+k\pi),\\
 \end{split}
\end{equation*}
where $n,k\in\mathbb{Z}$. We study the linearized system at a neighbourhood of the equilibrium points. At the points $P_1$ we have
\begin{equation}\label{l1}
\frac{\partial(h_1,h_2)}{\partial(\psi,\theta)}(P_1)=\begin{pmatrix}-1-\alpha&\alpha\\ 1&-1\end{pmatrix}.
\end{equation}
According to the eigenvalues, we have the  following possibilites. 
\begin{enumerate}
\item If $\alpha>0$, the two eigenvalues are negative  and $P_1$ are stable nodes.
\item If $\alpha\in (-2,0)$,  the eigenvalues are two complex numbers with negative real parts. Then $P_1$ are   stable spirals.
\item If $\alpha=-2$, then the eigenvalues are $\pm i$. Then $P_1$ is a center.
\item If $\alpha\in (-4,-2)$, the eigenvalues are two complex numbers with positive   real parts. Then $P_1$ are   unstable spirals.
 \item If $\alpha\leq -4$, then the eigenvalues are  two positive real numbers. Thus $P_1$ are  unstable nodes.
\end{enumerate}
For the points $P_2$, we have
\begin{equation}\label{l2}
\frac{\partial(h_1,h_2)}{\partial(\psi,\theta)}(P_2)=\begin{pmatrix}1+\alpha&-\alpha\\ -1&1\end{pmatrix}.
\end{equation}
The type of equilibrium points are: 
\begin{enumerate}
\item If $\alpha>0$, the two eigenvalues are positive   and    $P_2$ are unstable nodes.
\item If $\alpha\in (-2,0)$,  the eigenvalues are two complex numbers where the real parts are positive. Then $P_2$ are   unstable spirals.
\item If $\alpha=-2$, then the eigenvalues are $\pm i$. Then $P_2$ is a center.
\item If $\alpha\in (-4,-2)$, the eigenvalues are two complex numbers with negative   real parts. Then $P_2$ are   stable spirals.
\item If $\alpha\leq -4$, then   the eigenvalues are negative and $P_2$ are stable nodes. 
 \end{enumerate}
 
 The points $P_3$ have the same behaviour regardless of the value of $\alpha$ because
\begin{equation}\label{l3}
\frac{\partial(h_1,h_2)}{\partial(\psi,\theta)}(P_3)=\begin{pmatrix}-1&-\alpha\\ 0&1\end{pmatrix}
\end{equation}
and the eigenvalues are $-1$ and $1$. Hence $P_3$ are   saddle points with stable manifold in the direction of   $(1,0)$ and unstable manifold $V_1$ in the direction of   $(- \alpha ,2)$.

We begin the discussion with the case $\alpha>0$. 

\begin{theorem} \label{t2}
Suppose $\alpha>0$ and let  $\gamma(s)=(x(s),0,z(s))$, $s\in I$, denote a maximal solution   of \eqref{eq2}-\eqref{ini}. Then     $\gamma$ is a graph on the  $x$-axis.  In consequence, the axisymmetric $\alpha$-stationary surface is a entire graph on the plane $z=0$.
\end{theorem}

\begin{proof}
Since  $\psi'(0)=\frac{\alpha}{2}>0$, the point  $s=0$ is a local minimum of $z$ and  $\psi$ is increasing at   $s=0$.   We claim that  the function $\psi$ cannot increase until to attain the value $\pi/2$. Indeed, if    $s_1$ is the first point with $\psi(s_1)=\pi/2$, then $\psi'(s_1)\geq 0$ but \eqref{eq2}   gives 
$$\psi'(s_1)=-\alpha\frac{x(s_1)}{x(s_2)^2+z(s_1)^2}-\frac{1}{x(s_1)}<0.$$
A similar argument proves that there do not exist $s_1>0$ such that $s\to s_1$ and $\psi(s)\to\pi/2$ while $\lim_{s\to s_1}x(s)<\infty$. 

On the other hand, after $s=0$, the function $\psi$ cannot decrease and attain the value $0$ because if  $s_2>0$ if the first time where $\psi(s_2)=0$, then we have $\psi'(s_2)\leq 0$. However \eqref{eq2} gives 
$$\psi'(s_2)=\alpha\frac{z(s_2)}{x(s_2)^2+z(s_2)^2}>0.$$
This     contradiction proves that the function $\psi$ moves on the interval $(0,\frac{\pi}{2})$. Thus the function $z(s)$ is a graph on the $x$ axis and consequently, $\gamma$ is a graph on the $x$-axis. Moreover, since $s=0$ is the only critical point of $z$,    the point    $s=0$ is a global minimum of $z(s)$. See Fig. \ref{fig5}, right.

In Fig. \ref{fig5}, left, we plot the (black)  trajectory in the phase plane \eqref{eq3} going from  $P_3=(0,\frac{\pi}{2})$ to the point $P_1=(0,0)$. The point $P_3$ is a saddle point and   we know $\psi(s)>0$ and $\theta(s)<\pi/2$ around $s=0$. Thus the trajectory starts in the direction of the unstable manifold $V_1$ and entering in the fourth quadrant in the $(\psi,\theta)$-phase plane with respect to the point $P_3$. The direction of the trajectory is along the direction   $(\alpha,-2)$.  We see that the final point of the trajectory is  $P_1=(0,0)$, which it is stable node.
Let $P_2=(\pi,0)$ and $P_4=(\pi,\frac{\pi}{2})$ the other equilibrium points on the right hand-side of the $(\psi,\theta)$-phase plane. The angle $\varphi$ of the straight line $P_3P_2$  has tangent $\tan\varphi=\frac{\pi}{\pi/2}=2$. Since the unstable direction $V_1$ has tangent $\frac{2}{\alpha}$ and   $\alpha>0$, then $\frac{2}{\alpha}<2=\tan\varphi$.  This proves that  the trajectory corresponding to $\gamma$ lies  situated between the vertical  line $P_3P_1$ and the line $P_3P_2$: see Fig. \ref{fig5}, left. Finally the point $P_1$ is a stable node and the trajectory ends at $P_1$. 
 
\end{proof}
\begin{figure}[hbtp]
\begin{center}
\includegraphics[width=.5\textwidth]{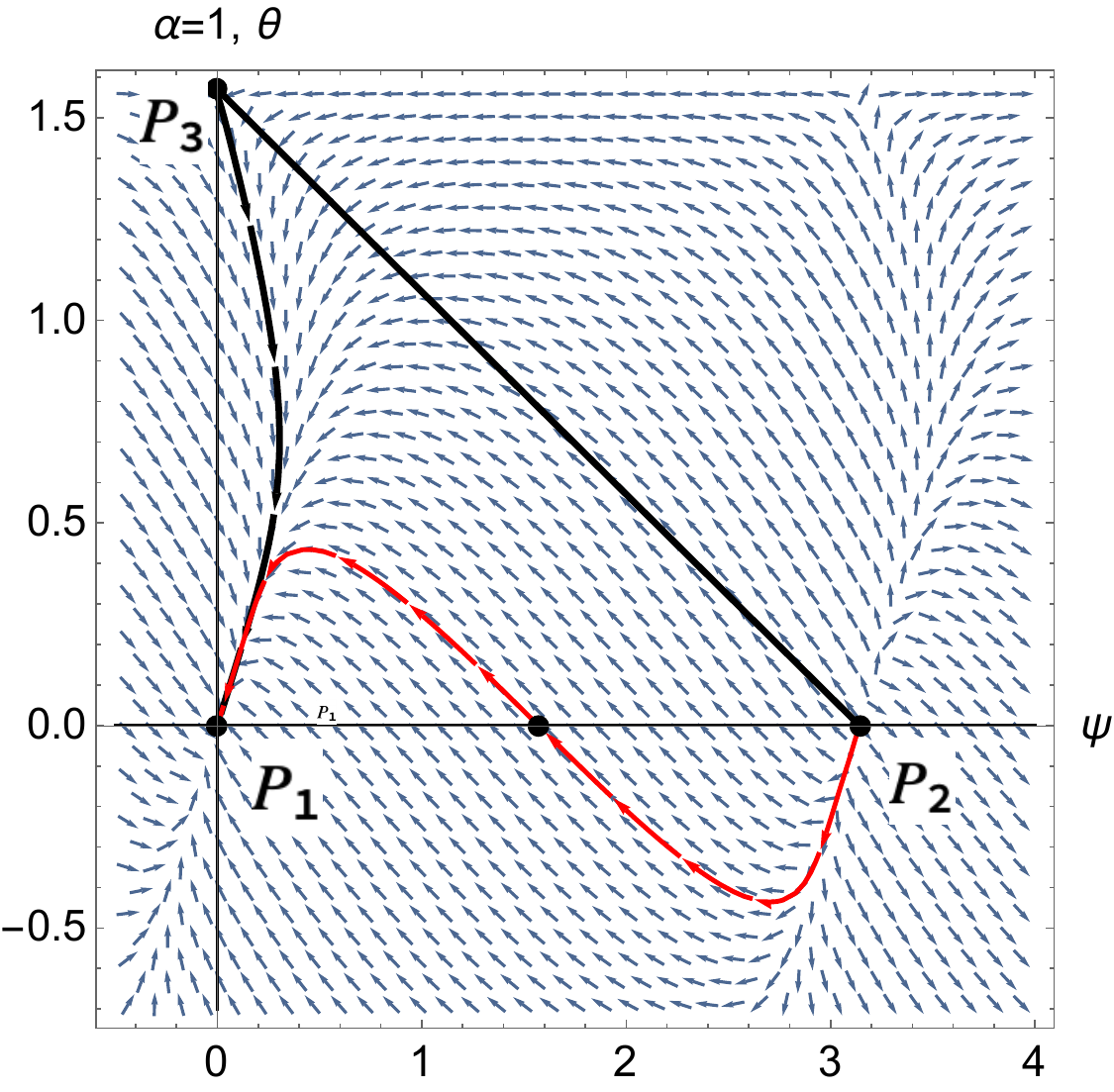} \quad \includegraphics[width=.45\textwidth]{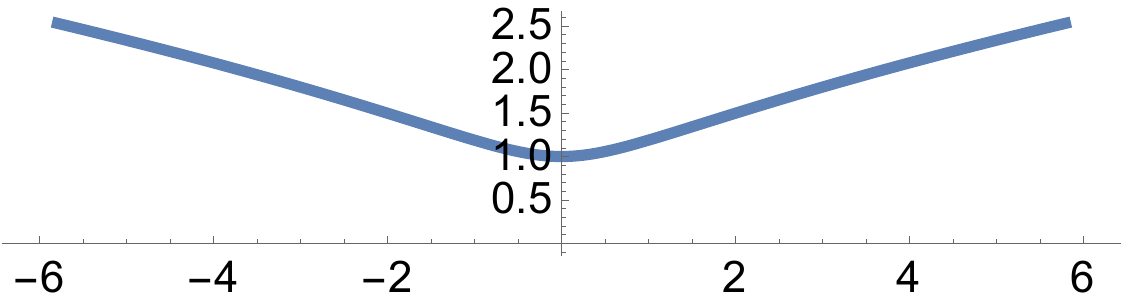} 
\end{center}
\caption{Case $\alpha>0$. Here $\alpha=1$.  The $(\psi,\theta)$-phase plane of \eqref{eq3} (left) for $\alpha=1$. Solutions of Eq. \eqref{rot-r} when   $\gamma$ intersects the $z$-axis.   }\label{fig5}
\end{figure}

We now consider the case $\alpha\in (-2,0)$. 

\begin{theorem} \label{t3}
Suppose $\alpha\in (-2,0)$ and let  $\gamma(s)=(x(s),0,z(s))$, $s\in I$, denote a maximal solution   of \eqref{eq2}-\eqref{ini}. Then   $\gamma$ goes to infinity oscillating along the $x$-axis. Moreover, outside a compact set around $0$, the curve $\gamma$ is a graph on the $x$-axis. The corresponding axisymmetric surface is a graph on the plane $z=0$ outside a compact containing $0$.  
\end{theorem}

\begin{proof}
The function $\psi$ is decreasing at $s=0$ because $\psi'(0)=\frac{\alpha}{2}<0$.  Since $\psi$ is negative for $s$ close to $0$, the trajectory corresponding to the solution $\gamma$ starts in the direction the unstable manifold $V_1$ at $P_3=(0,\frac{\pi}{2})$. This direction is $(\alpha,-2)$ and the trajectory lies in the third quadrant in the phase plane around $P_3$. As in the proof of Thm. \ref{t2}, consider the equilibrium points on the left hand-side of the $(\psi,\theta)$-phase plane: $P_2=(-\pi,0)$ and $P_4=(-\pi,\frac{\pi}{2})$. By using that $\alpha>-4$, a similar argument  proves that the trajectory lies between the lines $P_3P_1$ and $P_3P_2$ close to the point $P_3$: see Fig. \ref{fig4}, left. Since $P_1$ which it is a stable spiral node, the trajectory must end at $P_1$. As $s\to\infty$, we have $(\psi(s),\theta(s))\to (0,0)$ and this implies that $\gamma$ oscillates around the $x$-axis. Finally the $\psi$-coordinate of the trajectory goes to $0$. Thus $x'(s)=\cos\psi(s)\not=0$ as $s\to\infty$. This proves that $\gamma$ is a graph on the $x$-axis outside a compact containing $s=0$.

\end{proof}
\begin{figure}[hbtp]
\begin{center}
\includegraphics[width=.45\textwidth]{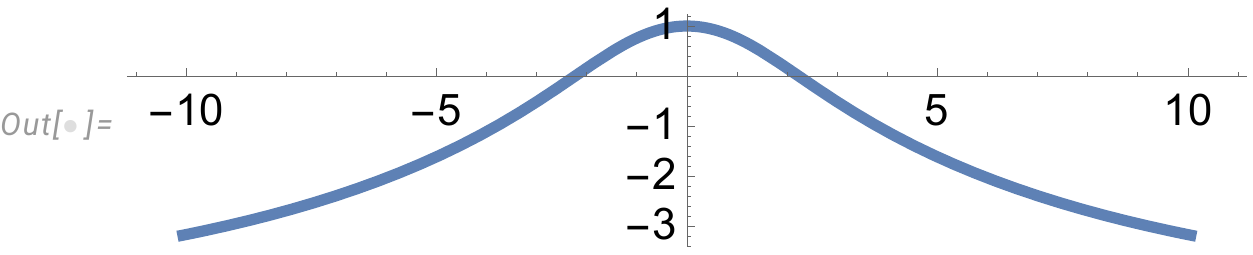}, \includegraphics[width=.5\textwidth]{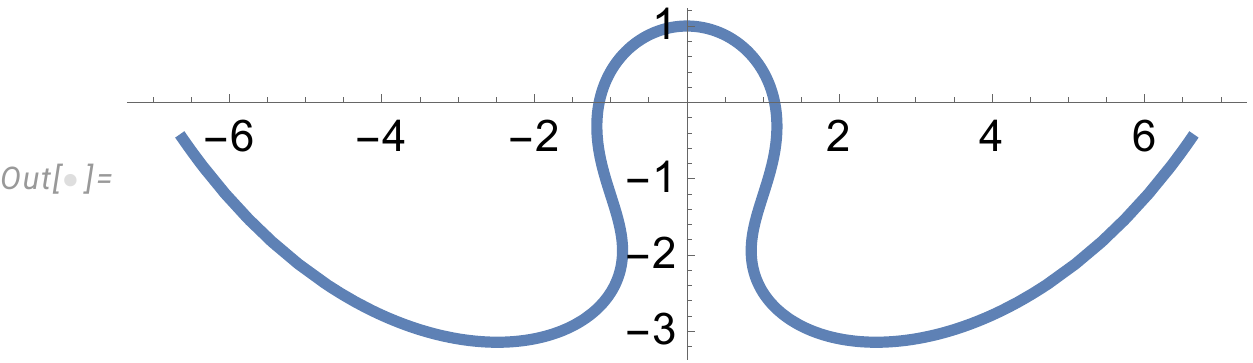}
\end{center}
\caption{Solutions of Eq. \eqref{rot-r} when $\alpha\in (-2,0)$. Here $\alpha=-1$ (left) and $\alpha=-1.8$ (right). }\label{fig4}
\end{figure}

 \begin{figure}[hbtp]
\begin{center}
\includegraphics[width=.4\textwidth]{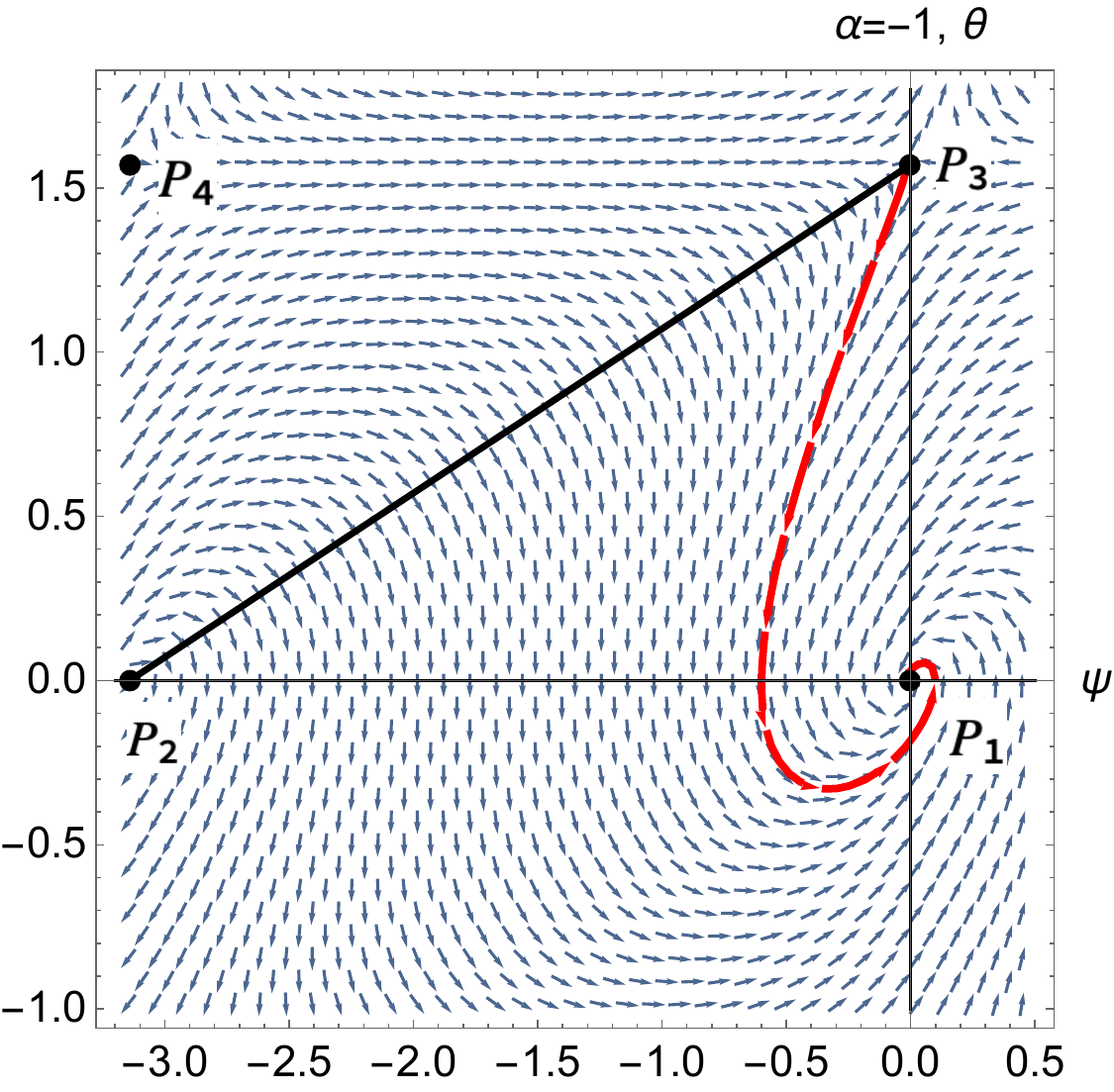}, \includegraphics[width=.4\textwidth]{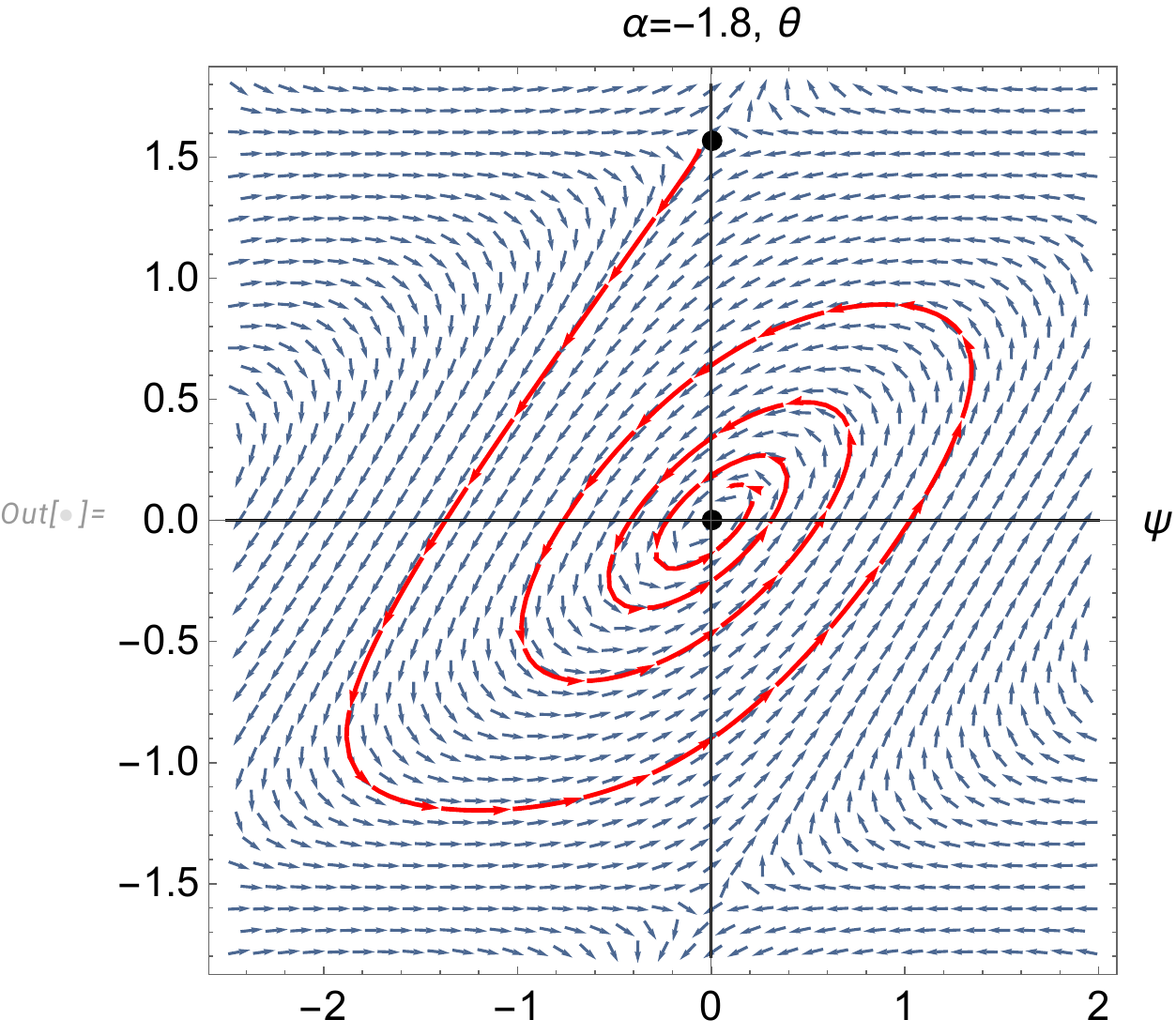}
\end{center}
\caption{The $(\psi,\theta)$-phase plane of \eqref{eq3} for $\alpha=-1$ (left) and $\alpha=-1.8$ (right).    }\label{fig-fase12}
\end{figure}

Finally, we consider the case $\alpha\leq -2$.  
 
\begin{theorem}\label{t5}
 Suppose $\alpha\leq -2$ and let  $\gamma(s)=(x(s),0,z(s))$, $s\in I$, denote a maximal solution   of \eqref{eq2}-\eqref{ini}.    If $\alpha<-2$, $\gamma(I)\cup\{0\}$ is an embedded closed curve non-smooth at $0$ except when $\alpha=-4$. Moreover, 

 \begin{enumerate}
 \item If $\alpha=-2$, then  $\gamma$ describes  a circle  centered at $0$. 
 \item If $\alpha\in (-4,-2)$, then $\gamma$ has points in both  half-planes $z>0$ and $z<0$. \item If $\alpha=-4$, then  $\gamma$  describes a circle of radius $\frac12$ centered at $(0,\frac12)$.
\item If $\alpha<-4$, then  $\gamma$ stays on the half-plane $z>0$ and $\gamma$ is a bi-graph on the $z$-axis. In particular, $\gamma(I)\cup\{0\}$ is an embedded closed curve.
\end{enumerate}
\end{theorem}

\begin{proof}
The cases $\alpha=-2$ and $\alpha=-4$ are straightforward  by uniqueness of solutions of \eqref{eq2}-\eqref{ini}. Suppose $\alpha\not= -2,-4$. We know that $0\in\overline{\Sigma}$ by Cor. \ref{cor1}. In consequence, by the symmetries of the solutions of \eqref{eq2}-\eqref{ini} with respect to the $z$-axis, these solutions are defined at some finite interval $[0,s_1)$, where $\lim_{s\to s_1}\gamma(s)=(0,0)$.

At $s=0$, we know $\psi'(0)=\alpha/2$ and $\psi=\psi(s)$ is a decreasing function at $s=0$. From \eqref{polar}, at the initial point $P_3$, we have $(\psi,\theta)=(0,\frac{\pi}{2})$ and $x(s)>0$ close to $s=0$. Therefore the corresponding trajectory    $\beta$   of $\gamma$ in the phase plane starts from $P_3$ going into the third quadrant with respect to the point $P_3$.  The trajectory $\beta$ corresponds to an unstable trajectory along  the direction $(\alpha,-2)$. Let denote the equilibrium points  
$$P_1=(0,0),\quad P_2=(-\pi,0),\quad P_3=(0,\frac{\pi}{2}),\quad P_4=(-\pi,\frac{\pi}{2}).$$

We use the same arguments of Thm. \ref{t2} and \ref{t3}. The angle $\varphi$ of the straight line $P_2P_3$  has tangent $\tan\varphi=\frac{\pi/2}{\pi}=1/2$. In consequence, the trajectory $\beta$ lies situated between the vertical  line $P_1P_2$ and the line $P_2P_3$ if $\frac{2}{-\alpha}>\frac12$, that is, if and only if $\alpha\in (-4,-2)$ close to $P_3$. Similarly, the trajectory $\beta$ lies situated between the horizontal line $P_3P_4$ and the line $ P_2P_3$ if   $\alpha<-4$: see Figs. \ref{fig-fase0} left and right, respectively. 

\begin{enumerate}
\item Case $\alpha\in (-4,-2)$. We know that the trajectory $\beta$ of the solution $\gamma$ starts between the lines $P_1P_3$ an $P_2P_3$. The equilibrium point $P_1$ is an unstable point hence $\beta$ cannot arrive to the point $P_1$. Thus $\beta$ goes to the point $P_2$. The trajectory $\beta$ crosses the line $P_1P_2$ with tangent vector $(-(1+\alpha)\sin\psi,\sin\psi)$ and the point $P_2$ is a stable spiral. In particular, $\theta$ is negative and this implies that the curve $\gamma$ crosses infinite times the $x$-axis as $\gamma$ goes to $(0,0)$. See Fig. \ref{fig3}, left. 

\item Case $\alpha<-4$. Close to the point $P_3$, the trajectory lies contained between the line $P_2P_3$ and $P_3P_4$. Since the flow of trajectories crosses the vertical line $P_2P_4$ inside the triangle $P_2P_3P_4$, then the trajectory $\beta$ corresponding to $\gamma$ must remain in that  triangle. In consequence, $\beta$ arrives to the equilibrium point $P_3$, which it is a stable node. Since $\beta$ moves in the triangle $P_2P_3P_4$, the $\theta$-coordinate of $\beta$ is positive which implies that the solution curve $\gamma$ is situated at the upper half-plane $z>0$. Since this also occurs for the $\psi$-coordinate, then $z'(s)=\sin\psi(s)\not=0$ and thus $\gamma$ is a graph on the $z$-axis. See Fig. \ref{fig3}, right. 

\end{enumerate}
  \begin{figure}[hbtp]
\begin{center}
\includegraphics[width=.45\textwidth]{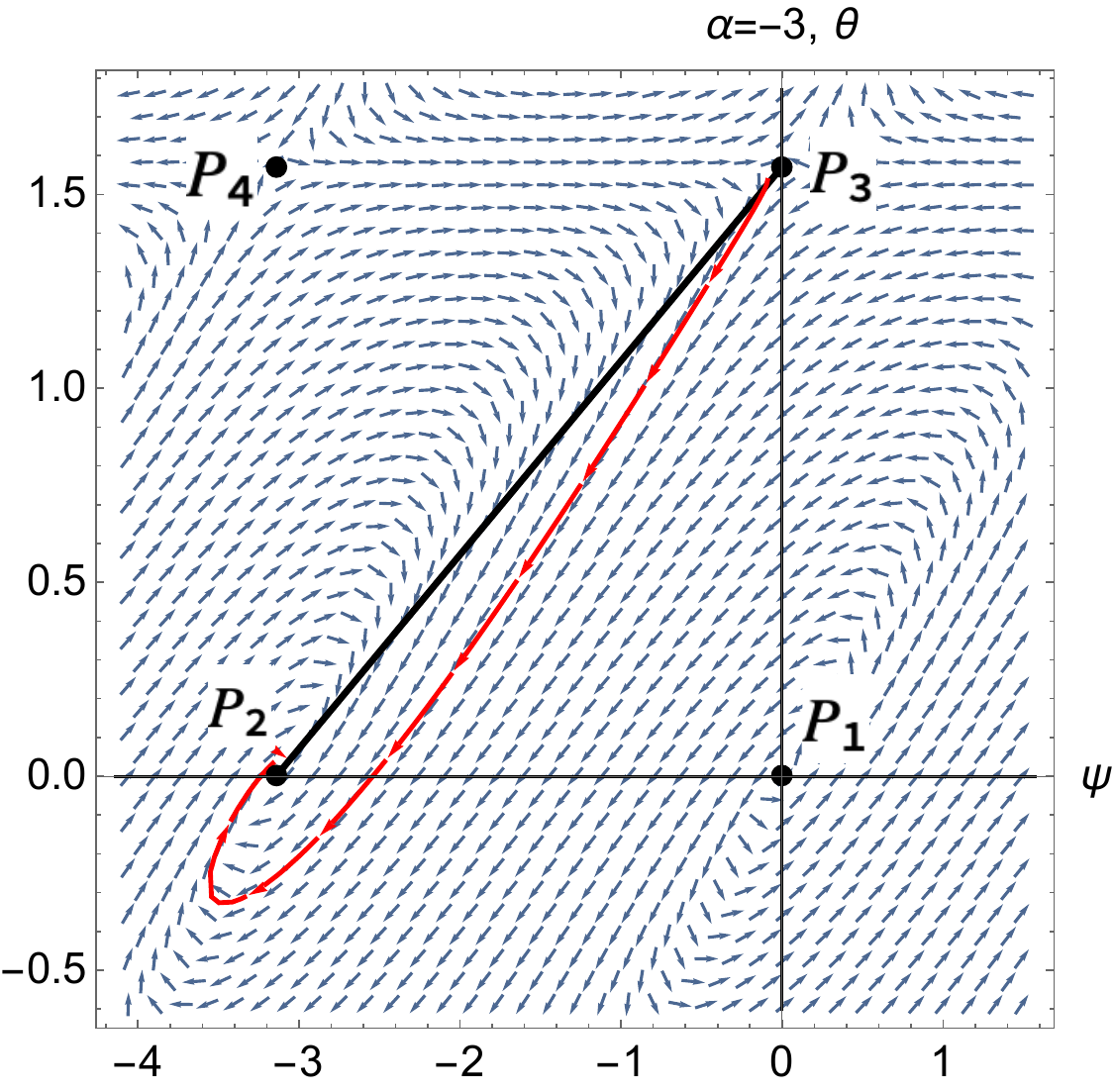}, \includegraphics[width=.45\textwidth]{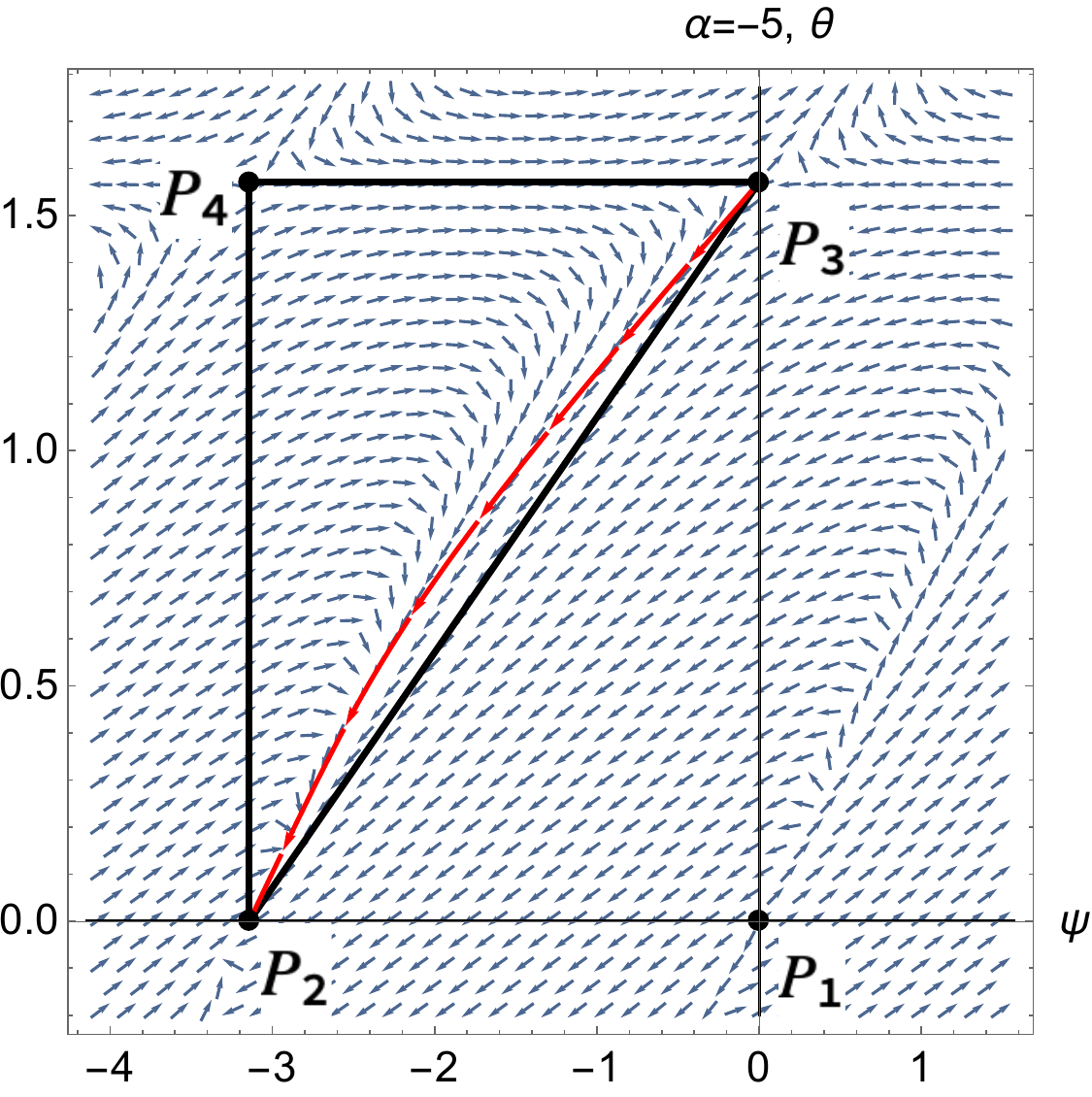}
\end{center}
\caption{The $(\psi,\theta)$-phase plane of \eqref{eq3} for $\alpha=-3$ (left) an $\alpha=-5$ (right).     }\label{fig-fase0}
\end{figure}
\begin{figure}[hbtp]
\begin{center}
\includegraphics[width=.45\textwidth]{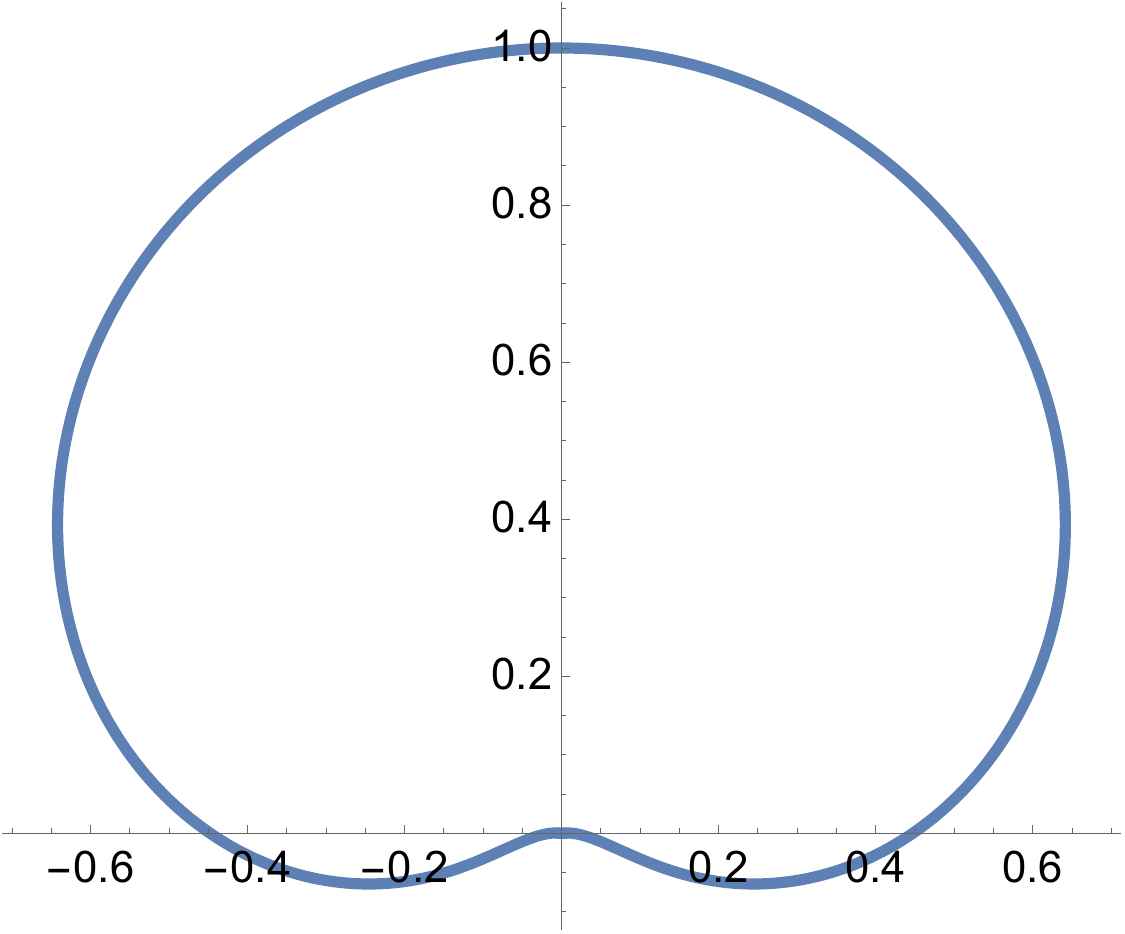}, \includegraphics[width=.35\textwidth]{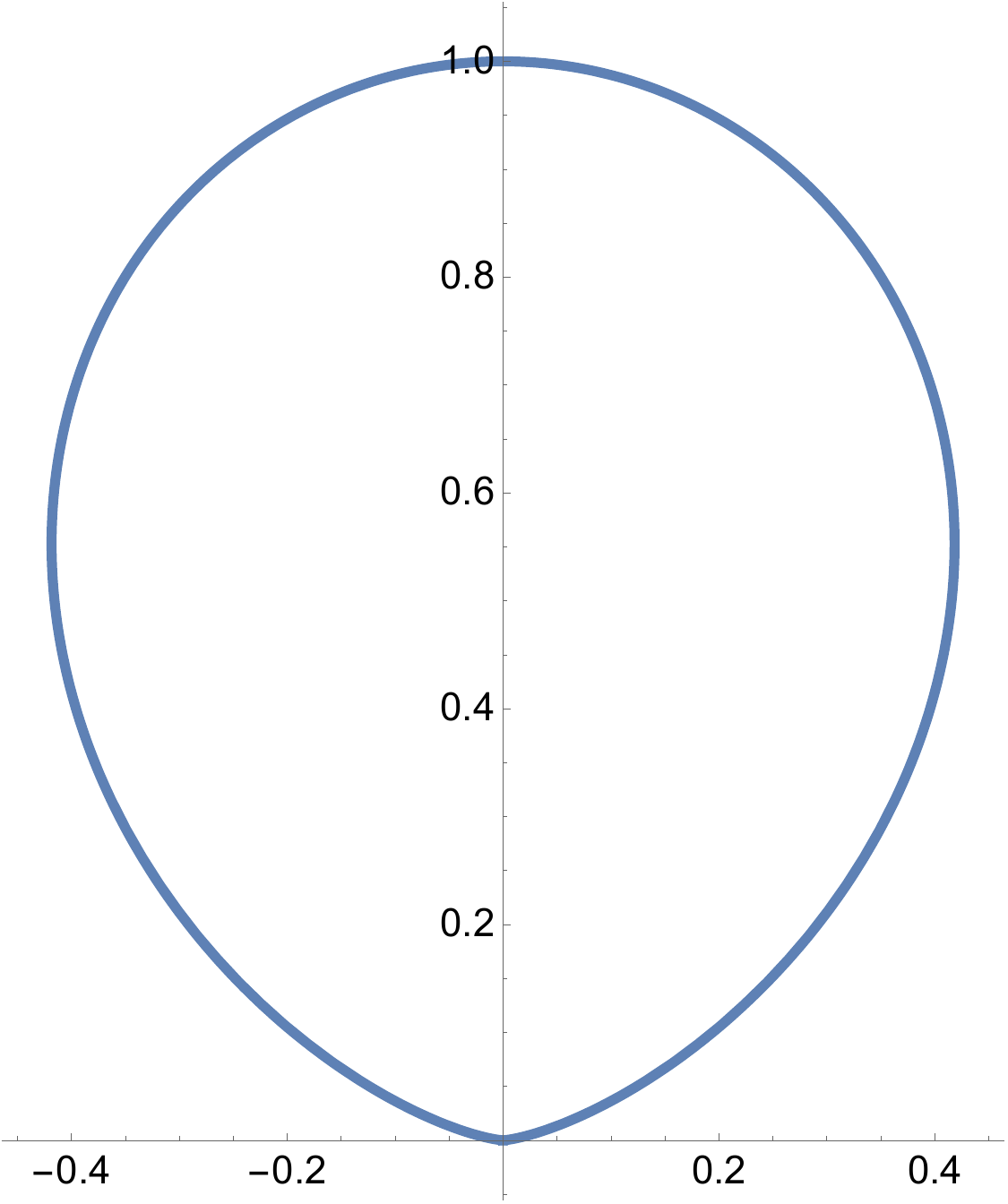}
\end{center}
\caption{Solutions of Eq. \eqref{rot-r} when  $\alpha=-3$ (left)  and $\alpha=-5$ (right).   }\label{fig3}
\end{figure}

\end{proof} 

\begin{remark} 
If $\alpha=-4$, Theorem \ref{t5} asserts that the solution is a circle of radius $\frac12$ centered at $(0,\frac12)$.  However, not any sphere crossing the origin, which it is a $-4$-stationary surface, is axisymmetric about the $z$-axis.
\end{remark}

The  axisymmetric stationary surfaces intersecting orthogonally the rotation axis  are determined by  the initial conditions \eqref{ini}.   In the final part of this section, consider the   interesting case that  the surface intersects orthogonally the plane $z=0$. For this, consider   initial conditions
\begin{equation}\label{ini2}
 \left\{
 \begin{split}
 (x(0),z(0))&=(1,0),\\
 \psi(0)&=\frac{\pi}{2}.
 \end{split}\right.
 \end{equation}
The condition $\psi(0)=\frac{\pi}{2}$ implies that $\gamma$ meets orthogonally the $x$-axis at the initial point $s=0$. The condition $x(0)=1$ is not restrictive because Eq. \eqref{eq2} is invariant by dilations. A first result is that  the corresponding stationary surface is invariant by reflections about the plane $z=0$.   
  
 \begin{proposition} Solutions of \eqref{eq2}-\eqref{ini2}   are symmetric about the $x$-axis.
 \end{proposition}
 \begin{proof}
 The result follows because if  $\{x,z,\psi\}$ is a solution of \eqref{eq2}-\eqref{ini2}, then $\{\overline{x}(s)=x(-s),\overline{z}(s)=-z(-s),\overline{\psi}(s)=\pi-\psi(-s)\}$   is also a solution of \eqref{eq2}-\eqref{ini2}. Uniqueness of ODEs proves symmetry with respect to $x$-axis.
 \end{proof}

In Figs. \ref{fig61} and \ref{fig62}, we have numerically calculated  the solutions of \eqref{eq2}-\eqref{ini2} for different values of $\alpha$. We analyze the case $\alpha>0$. The trajectory in the phase plane (red curve in Fig. \ref{fig5})  corresponding to the solution $\gamma$ crosses the point $(\frac{\pi}{2},0)$ and it goes  from   $P_2=(\pi,0)$ (unstable node) to $P_1$ (stable node). Since  the function $\theta$ does not attain the values $\pm \frac{\pi}{2}$,  the solution does not intersect the $z$-axis: see Fig. \ref{fig61}, left. In Figs. \ref{fig61} and \ref{fig62}, we show solutions of \eqref{eq2}-\eqref{ini2} for other  values of $\alpha$. In the case $\alpha<-2$, the extended solution must go to $0$ by Cor. \ref{cor1}. Furthermore, since $\gamma$ is symmetric about the $x$-axis, this implies that the curve $\gamma$ is a closed curve with $0\in\gamma(I)$. The smoothness of $\gamma$ at $0$ is not assured. This situation is similar that appeared in the solutions given in Thm. \ref{t5}. For example, if $\alpha<-4$, the curve $\gamma$ has not self-intersections and the corresponding stationary surface is an embedded closed surface with a singularity at $0$.

\begin{figure}[hbtp]
\begin{center}
\includegraphics[width=.4\textwidth]{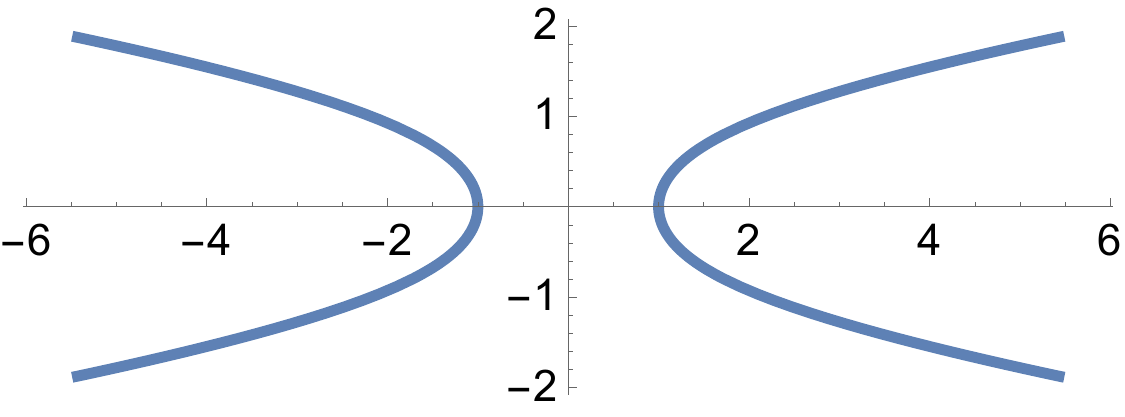}, \includegraphics[width=.5\textwidth]{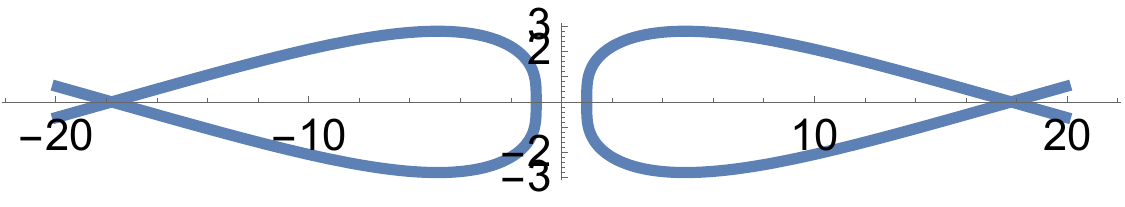}
\end{center}
\caption{Solutions of Eq. \eqref{eq2}-\eqref{ini2}. Cases  $\alpha=1$ (left) and $\alpha=-1$ (right).     }\label{fig61}
\end{figure}

\begin{figure}[hbtp]
\begin{center}
\includegraphics[width=.4\textwidth]{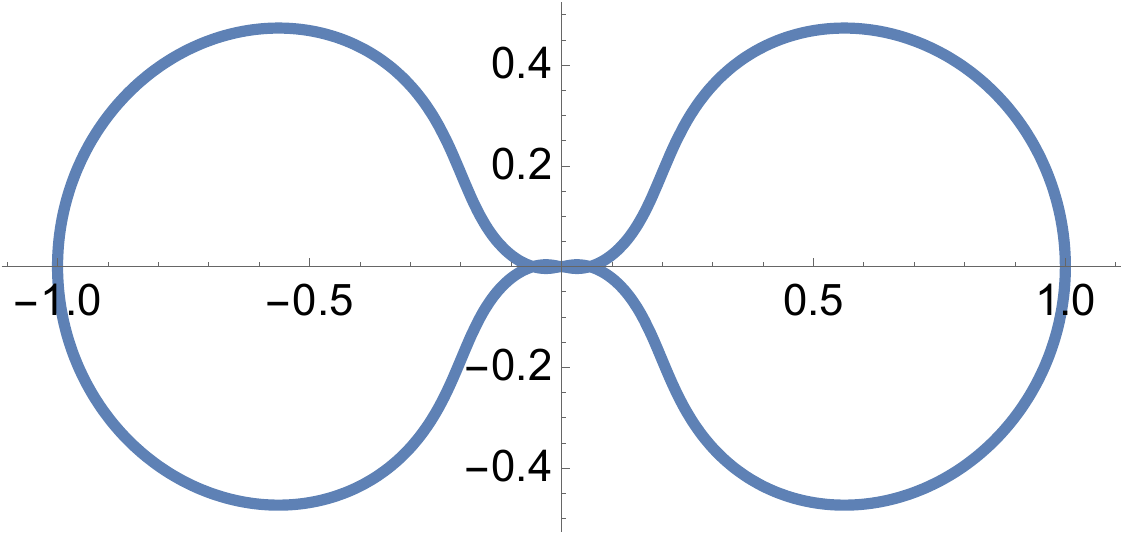}, \includegraphics[width=.5\textwidth]{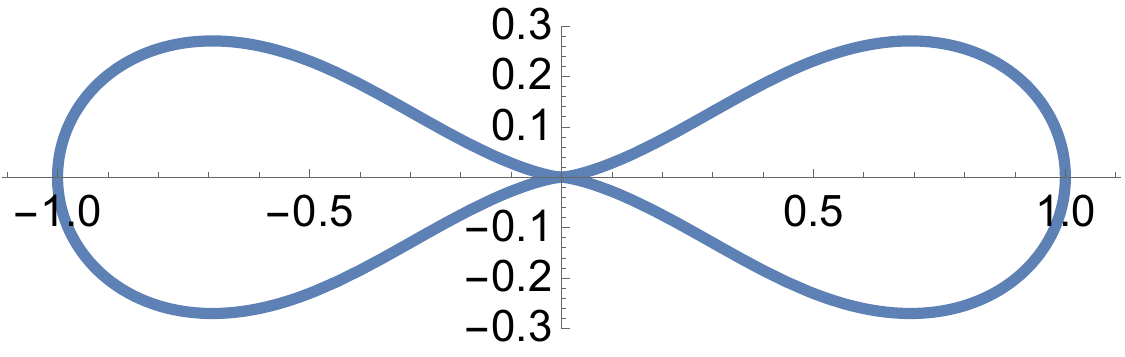}
\end{center}
\caption{Solutions of Eq. \eqref{eq2}-\eqref{ini2}. Cases  $\alpha=-3$ (left) and $\alpha=-5$ (right).     }\label{fig62}
\end{figure}

%%%%%%%%%%%%%%%%%%
\section{  Helicoidal  stationary surfaces}\label{s6}
%%%%%%%%%%%%%%%%%%%%%%%%%%%%%%%%

In this section we study   helicoidal stationary surfaces. A helicoidal motion of $\r^3$ is  a rotation about some line  followed by a translation parallel to that line. A  helicoidal surface $\Sigma$ in Euclidean space $\r^3$  is a surface    invariant under a one-parameter  group of helicoidal motions. The helicoidal group is determined by the twist axis and its pitch $h\in\r$. Without loss of generality, suppose that the twist axis is the $z$-axis. Then this group is $\{\mathcal{H}_t\colon t\in\r\}$ where $\mathcal{H}_t\colon\r^3\to\r^3$ is given by 
 \begin{equation}\label{h11}
 \mathcal{H}_t(x,y,z)= (x\cos t - y\sin t, x \sin t + y \cos t, z+ht).
  \end{equation}
 If $h=0$, then $\Sigma$ is simply a surface of revolution about the $z$-axis. In the following result we see that the  only helicoidal stationary   surfaces are rotational surfaces.  As in the case of rotational stationary surfaces, there is not an a priori relation between the twist axis of a helicoidal stationary surface and the origin $0$.

\begin{theorem}\label{t61}
 If $\Sigma$ is a   helicoidal stationary surface, then $h=0$, that is, $\Sigma$ is a surface of revolution. 
  \end{theorem}
  
  \begin{proof}
 The proof is by contradiction. Let $\Sigma$ be a  helicoidal stationary surface such that $h\not=0$.  Thanks to Prop. \ref{pr0},  and after a vector isometry of $\r^3$, we can assume that the twist axis $L$ is parallel to the $z$-axis and contained in the $xz$-plane. If the equations of $L$ are $\{x=q_1, y=0\}$, $q_1\geq 0$, the  helicoidal motions are as in  \eqref{h11} but moving the twist axis to $L$. Then  
 $$\mathcal{H}_t(x,y,z)=(q_1+(x-q_1)\cos t - y\sin t, (x-q_1)x \sin t + y \cos t, z+ht).$$
The generating curve  $\gamma\colon I\subset\r\to\r^3$ is a planar curve  contained in the $xz$-plane. Let $  \gamma(s)=(q_1+ x(s),0,z(s))$, where $x,z$ are smooth functions and $x(s)>q_1$. Then a  parametrization of $\Sigma$ is  
\begin{equation}\label{para}
\Phi(s,t)=(q_1+x(s)\cos t,x(s)\sin t,z(s)+ht).
\end{equation}
Without loss of generality, we can assume that $\gamma$ is parametrized by arc-length. Then $x'(s)=\cos\psi(s)$ and 
$z'(s) =\sin\psi(s)$ for some smooth function $\psi$. Regularity of $\Sigma$ is given by the condition $W:=x^2+h^2 \cos^2\psi>0$. The unit normal vector $\nu$ is 
\begin{equation}\label{normal}
\nu=\frac{1}{\sqrt{W}}\left(h\cos\psi\sin t-x\sin\psi\cos t,-h\cos\psi\cos t-x\sin\psi\sin t,x\cos\psi\right),
\end{equation}
and the mean curvature $H$  is
\begin{equation}\label{mean}
 H =  \frac{x (x^2+h^2) \psi '+\sin\psi  \left(x^2+2 h^2 \cos ^2\psi \right)}{ W^{3/2}}.
\end{equation}
After some computations, equation \eqref{eq1} becomes
\begin{equation}\label{h1}
A_0(s)+A_1(s)t+A_2(s)t^2+A_3(s)\sin t+A_4(s)\cos t=0,
\end{equation}
where $A_n=A_n(s)$ are smooth functions on the variable $s\in I$. Since the functions $\{1,t,t^2,\sin t,\cos t\}$ are linearly independent, then  all coefficients $A_n$ must vanish. A computation gives
$$A_3=\alpha h q_1\cos\psi W.$$
Equation $A_3=0$ gives two cases to discuss.
\begin{enumerate}
\item Case $q_1=0$. Then the twist axis is the $z$-axis. With this value of $q_1$, we have 
 $$A_2=h^2 \left(x  (x^2+h^2 ) \psi '+\sin \psi \left(x^2+2 h^2 \cos ^2\psi\right)\right).$$
Since $h\not=0$, we can get $\psi'$ from $A_2=0$, obtaining
$$\psi'=-\frac{\sin\psi(x^2+2h^2\cos^2\psi)}{x(h^2+x^2)}.$$
Substituting in $A_1$, we have 
$$A_1=\alpha h x\cos\psi W.$$
Equation $A_1=0$ yields  $\cos\psi=0$ identically. Since $x'(s)=\cos\psi(s)$, we deduce that  the function $x=x(s)$ is constant. If $x(s)=x_0$, then $A_2=x_0^2h^2$ and $A_2=0$ gives a contradiction because $x_0,h\not=0$.
\item Case $q_1\not=0$. Then $A_3=0$ implies  $\cos\psi=0$ identically and thus $x=x(s)$ is a constant function,  $x(s)=x_0$. Now we have  $A_2=x_0^2h^2$ and $A_2=0$ gives a contradiction again.
\end{enumerate}
  
 \end{proof}
 
Equation \eqref{eq1} for stationary   surfaces is  similar to the equation of the shrinkers of the mean curvature flow (MCF), which are obtained with a similar density. To be precise,  a  surface $\Sigma$ in $\r^3$ is said to be a {\it shrinker} of the MCF if $\Sigma$ is a weighted minimal surface for the density $\phi(p)=e^{\alpha|p|^2/2}$, where $\alpha\in\r$. Equation $H_\phi=0$ writes as $H=\alpha\langle  \nu,p\rangle$.    If $\alpha<0$, the surface is a called a self-shrinker while if $\alpha>0$, the surface is called a self-expander (\cite{eh}).  

 Shrinkers of helicoidal type form a rich family of surfaces (\cite{ha,km}). This  contrasts with the situation of stationary surfaces for $E_\alpha$ where the only helicoidal stationary surfaces are rotational surfaces (Thm. \ref{t61}).  As in the situation of stationary surfaces,   there is not an a priori relation between the twist axis of a   helicoidal shrinker and the origin of $\r^3$. However, as far as the authors know, there is not a discussion on it in the literature, and usually is assumed that the   twist axis of a helicoidal shrinker contains the origin. The following results clarifies this situation.    
 
 \begin{theorem} If $\Sigma$ is helicoidal  shrinker, then the twist axis crosses the origin of coordinates.
 \end{theorem}
 
 \begin{proof} The proof is similar to Thm. \ref{t61} and we will follow the same notation. The solutions of $H=\alpha\langle \nu,p\rangle$ are also invariant by vector isometries of $\r^3$. Thus we can assume that the       twist axis is $L=\{x=q_1, y=0\}$. The proof consists in proving that $q_1=0$. Now the equation $H=\alpha\langle \nu,z\rangle$ becomes as \eqref{h1} with $A_2$ trivially $0$. We distinguish if $h=0$ or $h\not=0$. 
 
\begin{enumerate}
\item Case $h\not=0$. We have 
 $$A_1=\alpha h x\cos\psi W.$$
 Then $A_1=0$ implies $\cos\psi=0$ identically and thus $x(s)=x_0$ is a constant function, $x_0>0$. Now $A_3$ is trivially $0$ and $A_4=\alpha q_1 x_0^3$ which implies $q_1=0$. 
\item Case $h=0$. Now the surface is rotational and   the equation $H=\alpha\langle\nu,p\rangle$ is $A_0(s)+A_4(s)\cos t=0$, where 
$A_4=\alpha q_1 x^3\sin\psi$. Equation $A_4=0$ gives two cases to discuss, namely,   $q_1=0$ and $\sin\psi=0$ identically. If $q_1=0$, the result is proved. Suppose $q_1\not=0$. Then   $\psi=0$ identically. Thus $z$ is a constant function $z(s)=z_0$. Now $A_0=\alpha x^3z_0$. This implies $z_0=0$ and the surface is the horizontal plane of equation $z=0$. This plane is also a surface of revolution about the $z$-axis, proving the result in this situation. 
\end{enumerate}
 
 \end{proof}

\section*{Acknowledgements}
Rafael L\'opez has been partially supported by MINECO/MICINN/FEDER grant no. PID2023-150727NB-I00,  and by the ``Mar\'{\i}a de Maeztu'' Excellence Unit IMAG, reference CEX2020-001105- M, funded by MCINN/AEI/10.13039/ 501100011033/ CEX2020-001105-M.
%%%%%%%%%%%%%%%%%%%%%%%%%%%%%%%%%%%%%%%%%%%%%%%%%%%%%%%%%%%%%%

\end{document}